\documentclass[a4paper,9pt]{article}
\usepackage{amsfonts}
\usepackage{amssymb,amsthm,color,comment,graphicx}
\usepackage{amsmath}
\usepackage[normalem]{ulem}

\newcommand{\R}{\mathbb{R}}
\newcommand{\N}{\mathbb{N}}

\newtheorem{theorem}{Theorem}
\newtheorem{lemma}[theorem]{Lemma}

\newtheorem{definition}{Definition}
\newtheorem{remark} {Remark}

\newcommand{\Om}{\Omega}
\newcommand{\lb}{\lambda}

\newcommand{\sm}{\setminus}

\newcommand{\sq}{\subseteq}
\newcommand{\ov}{\overline}

\newcommand{\vps}{\varepsilon}

\newcommand{\bp}{\begin{proof}}
\newcommand{\ep}{\end{proof}}

\begin{document}
\title{Sign changing solutions of Poisson's equation}

\author{{M. van den Berg} \\
School of Mathematics, University of Bristol\\
Fry Building, Woodland Road\\
Bristol BS8 1UG\\
United Kingdom\\
\texttt{mamvdb@bristol.ac.uk}\\
\\
{D. Bucur}\\
Laboratoire de Math\'ematiques, Universit\'e Savoie Mont Blanc \\
UMR CNRS  5127\\
Campus Scientifique,
73376 Le-Bourget-Du-Lac\\
France\\
\texttt{dorin.bucur@univ-savoie.fr}}
\date{20 January 2020}\maketitle
\vskip 1.5truecm \indent
\begin{abstract}\noindent
Let $\Omega$ be an open, possibly unbounded, set in Euclidean space $\R^m$   with boundary $\partial\Omega,$   let $A$ be a measurable subset of $\Omega$  with measure $|A|$, and let $\gamma \in (0,1)$.
We investigate whether the solution $v_{\Om,A,\gamma}$ of $-\Delta v=\gamma{\bf 1}_{\Omega \setminus  A}-(1-\gamma){\bf 1}_{A}$ with    $v=0$ on $\partial \Omega$  changes sign.
Bounds are obtained for $|A|$ in terms of geometric characteristics of $\Om$ (bottom of the spectrum of the Dirichlet Laplacian, torsion, measure, or $R$-smoothness of the boundary) such that ${\rm essinf} v_{\Om,A,\gamma}\ge 0$. We show that ${\rm essinf} v_{\Om,A,\gamma}<0$ for any measurable set $A$, provided $|A| >\gamma |\Om|$. This value is sharp. We also study the shape optimisation problem of the optimal location of $A$ (with prescribed measure) which minimises the essential infimum of $v_{\Om,A,\gamma}$. Surprisingly, if $\Om$ is a ball, a symmetry breaking phenomenon occurs.
\end{abstract}
\begin{center} \textbf{Keywords}: Torsion function, Dirichlet boundary condition, Poisson's equation.
\end{center}
\mbox{}
\section{Introduction\label{sec0}}
Let $\Omega$ be an open, possibly unbounded, set in Euclidean space $\R^m$ with boundary $\partial\Omega$, and with, possibly infinite, measure $|\Omega|$.
It is well-known \cite {vdBC} that if the bottom of the Dirichlet Laplacian defined by
\begin{equation*}
\lambda(\Omega)=\inf_{\varphi\in H_0^1(\Omega)\setminus\{0\}}\frac{\displaystyle\int_\Omega|D\varphi|^2}{\displaystyle\int_\Omega \varphi^2},
\end{equation*}
is bounded away from $0$, then
\begin{equation}\label{e1}
-\Delta v=1,\, v=0 \mbox { on } \partial \Omega,
\end{equation}
has a unique weak solution denoted by $v_{\Omega}$, which is non-negative, and which satisfies,
\begin{equation}\label{ee}
\lambda(\Omega)^{-1}\le \|v_{\Omega}\|_{ L^{\infty}(\Omega)}\le (4+3m\log 2)\lambda(\Omega)^{-1}.
\end{equation}
The $m$-dependent constant in the right-hand side of \eqref{ee} has been improved in \cite{GS}, and subsequently in \cite{HV}.

If $|\Om|<\infty$ then, by the Faber-Krahn inequality, $\lambda(\Om)>0$, and by \eqref{ee},  $v_{\Om}\in  H^1_0(\Omega)$,  and   $v_{\Om}\in L^1(\Om)$. For an arbitrary open set $\Om$ we define the torsion, or torsional rigidity, by
\begin{equation*}
T(\Om)=\int_{\Om}v_{\Om}.
\end{equation*}

Note that, under the assumption $\lambda(\Omega)>0$, by \eqref{ee} the solution of  an equation like in \eqref{e1} with a right-hand side $f\in L^\infty (\Omega)$ can be defined by approximation on balls for the positive and negative parts of $f$.

For a measurable subset $A\subset\Omega$, with $\lambda(\Om)>0$, and $0<\gamma<1$, we denote by $v_{\Omega,A,\gamma}$ the solution of
\begin{equation}\label{e2}
-\Delta v=\gamma {\bf 1}_{\Omega\setminus A}-(1-\gamma){\bf 1}_{A},\, v=0 \mbox { on } \partial \Omega.
\end{equation}
 These hypotheses on $A$, $\Om$ and $\gamma$ will not be repeated in the statements of all lemmas and theorems below.

This paper investigates whether the solution of \eqref{e2} satisfies ${\rm essinf} \; v_{\Om,A,\gamma}<0.$ Whether this holds depends on the geometry of $\Om$, and on the size and the location of the set $A\subset \Om$. This question shows up in a variety of situations. We refer, for instance,  to \cite{HKS}, where $v$ is a scalar potential and the right-hand side stands for a magnetic field which changes sign. The influence of the magnetic field  on the asymptotic behaviour of the bottom of the spectrum  of the Pauli operator is effective provided that the scalar potential has constant sign, that is, ${\rm essinf} \; v_{\Om,A,\gamma}=0$. In fluid mechanics, the function $v$ can be interpreted as a vorticity stream function, for a vorticity taking the values $\gamma$ and $-(1-\gamma)$. If $v_{\Om,A,\gamma}$ changes sign then there exist at least two stagnation points. More situations where the sign question of the state function is put in relationship with sign changing data can be found in  \cite{BBV}, \cite{FH}, \cite{McG} and, in some biological models, \cite{LLNP16}.

\begin{definition}\label{def1} For $\gamma \in (0,1)$, $\Om \subset \R^m$, with $\lambda (\Om) >0$,
\begin{equation*}
\mathfrak{C}_{-}(\Omega,\gamma)=\sup\{c\ge 0: \forall\, A\subset \Omega, A\, \textup {measurable}, |A|\le c, {\rm essinf} \; v_{\Omega,A,\gamma}\ge 0\},
\end{equation*}
\begin{equation*}
 \mathfrak{C}_{+}(\Omega,\gamma)=\inf\{c\ge 0: \forall\, A\subset \Omega, A\, \textup {measurable}, |A|> c, {\rm essinf} \; v_{\Om,A,\gamma}<0\}.
\end{equation*}
\end{definition}
It follows immediately from the definition that for a homothety $t\Om,\, t>0$ of $\Om$ we have the scaling relations
\begin{equation}\label{ee9}
\mathfrak{C}_{-}(t\Omega,\gamma)=t^m\mathfrak{C}_{-}(\Omega,\gamma),
\end{equation}
and
\begin{equation*}
\mathfrak{C}_{+}(t\Omega,\gamma)=t^m\mathfrak{C}_{+}(\Omega,\gamma).
\end{equation*}
Furthermore if $\Om_1, \Om_2$ are disjoint open sets, then
$$\mathfrak{C}_{-}(\Omega_1\cup \Om_2,\gamma)=\min \{ \mathfrak{C}_{-}(\Omega_1 ,\gamma), \mathfrak{C}_{-}(\Omega_2 ,\gamma)\},$$
$$\mathfrak{C}_{+}(\Omega_1\cup \Om_2,\gamma)=\mathfrak{C}_{+}(\Omega_1 ,\gamma) + \mathfrak{C}_{+}(\Omega_2 ,\gamma).$$

This paper concerns the analysis of these quantities and their dependence on  $\Om$. It turns out that $ \mathfrak{C}_{+}(\Omega,\gamma)=\gamma |\Om|$ for arbitrary open sets $\Om$ with finite measure. On the contrary, $\mathfrak{C}_{-}(\Omega,\gamma)$ is very sensitive to the geometry. We find its main properties, give basic estimates, establish isoperimetric and isotorsional inequalities, and we discuss the shape optimisation problem related to the optimal location of the set $A$ in order to minimise the essential infimum.
\begin{theorem}\label{the1}
{For every non-empty open set $\Om \sq \R^m$ of finite measure we have}
$$ \mathfrak{C}_{+}(\Omega,\gamma) = \gamma|\Om|.$$
\end{theorem}

Below we show that, in general, we have to assume some regularity of $\Om$ in order to have $\mathfrak{C}_{-}(\Omega,\gamma)>0$.
For instance, if $\Omega=\cup_{j\in \N}C_j$ is a set of finite measure, where the sets $C_j,j\in \N$ are non-empty, open, disjoint, then $\mathfrak{C}_{-}(\Omega,\gamma)=0$. Indeed, if we let  $A=C_j$ then  ${\rm essinf}\;v_{\Om,A,\gamma}\le (\gamma-1){\rm   esssup} \;v_{C_j}<0.$ Consequently,
$\mathfrak{C}_{-}(\Omega,\gamma) \le |C_j|$ for every $j$, so $\mathfrak{C}_{-}(\Omega,\gamma)=0.$

\begin{theorem}\label{the2} If  $\Omega\subset \R^2$ is any {open triangle,} then
 $\mathfrak{C}_{-}(\Omega,\gamma)=0$.
 \end{theorem}

In Theorem \ref{the3} below we show that if $\Om$ is bounded, and $\partial\Om$ is of class $C^2$ then $\mathfrak{C}_{-}(\Om)>0$. In order to quantify this assertion we
introduce some notation.  For a non-empty open set $\Omega$ we denote by $\textup{diam}(\Omega)=\sup\{|x-y|:\, x\in \Omega,\, y\in \Omega\}$.  We denote the complement $\R^m \sm E$ of $E$ by $E^c$, and the closure of $E$ by $\overline{E}$.
Furthermore, $B_r(x):=\{y\in \R^m: |x-y|<r\}$ denotes the open ball centred at $x$ of radius $r$. If $x=0$, we simply write $B_r$. We set $\omega_m=|B_1|$. For $x\in \Om$ we let $\bar{x}\in \partial \Om$ be a point such that $|x-\bar x|=\min\{|x-z|:z\in \partial\Omega\}$. We recall the following from \cite[p.280]{vdB}.
\begin{definition}\label{def2}
An open set $\Omega \subset \mathbb{R}^{m}$, $m\geq 2$, has $R$-smooth
boundary if at any point $x_{0} \in \partial \Omega$, there are two
open balls $B_R(x_1), \, B_{R}(x_2)$ such that $B_R(x_1) \subset \Omega$, $B_R(x_2)
\subset \mathbb{R}^{m}\sm \bar{\Omega}$ and $\bar B_R(x_1)  \cap
\bar B_R(x_2)=\{x_{0}\}$.
\end{definition}
We also recall that a bounded $\Om$ with $C^2$ boundary $\partial\Om$ is $R$-smooth for some $R>0$.
\begin{theorem}\label{the3}
If $\Om$ is an open, bounded set in $\R^m$ with a $C^2$ and $R$-smooth boundary, then
\begin{equation*}
\mathfrak{C}_{-}(\Om,\gamma)\ge\mathfrak{C}_{-}(B_R,\gamma).
\end{equation*}
Furthermore
\begin{equation}\label{e00}
\mathfrak{C}_{-}(B_R,\gamma)\ge \bigg(\frac{\gamma}{4m}\bigg)^m\omega_mR^m,\, m=2,3,...,
\end{equation}
\begin{equation}\label{ex1}
\mathfrak{C}_{-}(B_R,\gamma)\le \gamma^{m/2}\omega_mR^m, \, m\ge 3,
\end{equation}
and
\begin{equation}\label{ex2}
\mathfrak{C}_{-}(B_R,\gamma)\le \Big(1+\log \Big(\frac{1}{\gamma}\Big)\Big)^{-1}\gamma\pi R^2,\, m=2.
\end{equation}
\end{theorem}

The  following inequality gives an upper bound for $\mathfrak{C}_{-}(\Om,\gamma)$ in terms of $\lambda(\Om)$.
\begin{theorem}\label{the4}
For every open set $\Om \subset \R^m$ with $\lambda(\Om)>0$,
\begin{equation*}
 \mathfrak{C}_{-}(\Omega,\gamma) \le C_1(m)\bigg(\frac{\gamma}{1-\gamma}\bigg)^{m/2} \lambda(\Om)^{-m/2},
 \end{equation*}
 where
 \begin{equation}\label{ee11.0}
 C_1(m)=\omega_m^{(m+2)/2}2^{5m^2/12}3^{m(m+2)/4}e^{2^{1/m}\lambda(B_1)^{1/2}/24}\bigg(\frac{12m(m+2)}{eC_2(m)^{1/2}}\bigg)^{m(m+2)/2},
 \end{equation}
 and where $C_2(m)$ is the constant in the Kohler-Jobin inequality \eqref{ea1} below.
\end{theorem}
This implies that if $\Om$ is an open set with $T(\Om)<\infty$, then
\begin{equation}\label{ee12}
\mathfrak{C}_{-}(\Omega,\gamma) \le C_1(m)C_2(m)^{-m/2}\bigg(\frac{\gamma}{1-\gamma}\bigg)^{m/2}T(\Om)^{m/(m+2)}.
\end{equation}
The optimal coefficient of  $T(\Om)^{m/(m+2)}$ in \eqref{ee12} is not known. However, the Kohler-Jobin inequality suggests to prove (or disprove) optimality for balls.

\begin{theorem}\label{the4.1}
There exists $C_3(m)<\infty$ such that for every open, connected set $\Om \subset \R^m$ with $T(\Om)<\infty$,
\begin{equation}\label{ee11.1}
 \mathfrak{C}_{-}(\Omega,\gamma) \le C_3(m)\max\bigg\{\bigg(\frac{\gamma}{1-\gamma}\bigg)^{\frac m2},\bigg(\frac{\gamma}{1-\gamma}\bigg)^{\frac{m(m+2)}{2(m+1)}}\bigg\} \bigg(\frac{T(\Om)}{\textup{diam} (\Om)}\bigg)^{m/(m+1)}.
 \end{equation}
In particular, if $\Om$ is unbounded, then $ \mathfrak{C}_{-}(\Omega,\gamma)=0$.
The value of $C_3(m)$ can be read-off from the proof in Section \textup{\ref{sec4.1}}.
\end{theorem}

We see from Theorems \ref{the1} and \ref{the3} that $\mathfrak{C}_{-}(B_R,\gamma)<\mathfrak{C}_{+}(B_R,\gamma)$. The isoperimetric inequality below generalises this to arbitrary open sets with finite measure.
\begin{theorem}\label{the5}
\begin{equation}\label{ee11.a}
 \sup \bigg\{\frac{\mathfrak{C}_{-}(\Omega,\gamma)}{|\Om|} : \Om \sq \R^m, \mbox{ $\Om$  \textup{open}, { $0<|\Om|<\infty$}}\bigg\} <\gamma.
\end{equation}

\end{theorem}
The theorem above implies that $\mathfrak{C}_{-}(\Omega,\gamma)\le C(m, \gamma) |\Om|$ for every open set of finite measure, with $ C(m,\gamma)<\gamma$.
The proof of Theorem \ref{the5} relies on the relaxation of the shape optimisation problem \eqref{ee11.a} to the larger class of quasi open sets.  We  shall prove that the supremum is attained at some quasi open set $\Omega^*$ for which
$\mathfrak{C}_{-}(\Omega^*,\gamma)< \gamma |\Om^*|$.

The optimal value $ C(m,\gamma) = \frac{\mathfrak{C}_{-}(\Omega^*,\gamma)}{|\Om^*|} $ is not known, nor  whether $\Omega^*$   is open. The symmetry breaking phenomenon for balls   stated   in Theorem \ref{the6} below does not support the ball to be a maximiser.

Given a constant $c \in ( \mathfrak{C}_{-}(\Omega,\gamma), |\Om|)$, there exists at least one set $A \sq \Om$, $|A|=c$  such that  ${\rm essinf} \; v_{\Omega,A,\gamma}< 0$. A natural question is to find the best location of the set $A$ of measure $c$, which minimises ${\rm essinf} \; v_{\Omega,A,\gamma}$. This question is of particular interest for values of $c$ close to $\mathfrak{C}_{-}(\Omega,\gamma)$, as this gives information on where the geometry of $\Om$   is most sensitive to negative values. We prove the following shape optimisation result for the optimal location.

\begin{theorem}\label{the6} Let $\gamma \in (0,1)$ and let $\Om \subset \R^m$  be an open, bounded and connected set with a smooth boundary $\partial \Om$. For every $c \in ( \mathfrak{C}_{-}(\Omega,\gamma), |\Om|)$, the shape optimisation problem
\begin{equation}\label{e55}
\min \{ {\rm essinf} \; v_{\Om, A, \gamma } : A \subset \Om, |A|=c\},
\end{equation}
has a solution. Moreover, if $\Om$ is a ball $B$ then, depending on the value of $c$, the  optimal locations may be radial or not.
\end{theorem}
The existence of an optimal set relies partly on a concavity property of the shape functional $A \mapsto {\rm essinf} \; v_{\Om,A,\gamma}$. We point out that the proof relies on both the concavity, and the analysis of optimality conditions in relationship with the partial differential equation \eqref{e1} (see \cite{CL96}). If $\Om$ is a ball $B$ and $c$ is close to $|B|$, then the optimal location is a ball. If $c$ is close to $\mathfrak {C}_{-}(B,\gamma)$  then the optimal location is no longer radial.  This symmetry breaking phenomenon occurs at a value $c\in (\mathfrak {C}_{-}(B,\gamma), \gamma |B|)$, and is supported by analytical, and numerical computations.

Theorem \ref{the6} can be interpreted both as a (rather non-standard) shape optimisation problem or as an optimisation problem in a prescribed class of rearrangements, see, for example \cite{ALT89}. We also refer  to the paper of Burton and Toland \cite{BT11} for models of steady waves with vorticity, where the distribution of the vorticity is prescribed, but we point out that our problem is essentially of  different nature since the functional to be minimised  is not an energy of the problem.

 The proofs of Theorems \ref{the1}, \ref{the2}, \ref{the3}, \ref{the4}, \ref{the4.1}, \ref{the5} and \ref{the6} are deferred to Sections \ref{sec1}, \ref{sec2}, \ref{sec3}, \ref{sec4}, \ref{sec4.1}, \ref{sec5}, and \ref{sec6} below.

\section{Proof of Theorem \ref{the1} \label{sec1}}
In order to simplify notation, throughout the paper, if $\Om$ is an open set and $A\subset \R^m$ is measurable, not necessarily contained in $\Om$,  by $v_{\Om,  A, \gamma}$   we mean $v_{\Om, \Om\cap A, \gamma}$.

\begin{proof} Firstly assume that $\Om\subset \R^m$ is an open set with finite measure. Assume that $A \subset \Om$ is a measurable set such that $v_{\Om, A, \gamma} \ge 0$. In a first step, we shall prove  that $|A|\le \gamma |\Om|$. As a consequence,  $\mathfrak{C}_{+}(\Omega, \gamma)\le \gamma |\Om|$.

Indeed, since $v_{\Om, A, \gamma} \ge 0$, one can use Talenti's theorem (see for instance \cite[Theorem 3.1.1]{ke06}) in the following way. We denote  by $v^*$ the Schwarz rearrangement of $v_{\Om, A, \gamma}$, and by $f^*$ the rearrangement of $ \gamma1_{\Om\sm A}-(1- \gamma)1_A$. There exist two positive values $0<r_1<r_2$ such that
 $f^*=  \gamma1_{B_{r_1}}-(1- \gamma)1_{B_{r_2}\sm B_{r_1}}$, where $r_1$ is such that $|B_{r_1}|=|\Om \sm A|$ and $|B_{r_2}|=|\Om|$.   By   Talenti's theorem, we get
$$0\le v^* \le v_{B_{r_2}, B_{r_1}^c,  \gamma}.$$
By elementary  computations,  one gets the expression  for  $v_{B_{r_2},  B_{r_1}^c,  \gamma}$. Indeed, the solution
$v_{B_{r_2},  B_{r_1}^c, \gamma}$ is radially symmetric and satisfies the equation
$$-v''- \frac{m-1}{r} v' =  \gamma1_{[0,r_1]} - (1- \gamma)1_{[r_1,r_2]},$$
 with initial condition  $v'(0)=0,$ and $v(r_2)=0$. Moreover, the solution is $C^{1, \alpha}$ regular, for some $\alpha >0$.

We integrate separately on $[0,r_1]$, and on $[r_1,r_2]$, and write the equality of the left- and right-derivatives in $r_1$, namely $v'_-(r_1)=v'_+(r_1)$. Hence, we get
$$ - \gamma \frac{r_1}{m}= \frac{r_2^{m-1}}{r_1^{m-1}}v'(r_2)-(1- \gamma)\frac{r_2^m}{mr_1^{m-1}}+ (1- \gamma)\frac{r_1}{m}.$$
In general, from the positivity of $ v_{B_{r_2},  B_{r_1}^c, \gamma}$ one gets that $v'(r_2)\le 0$. Hence,
$$(1-\gamma)\frac{r_2^m}{mr_1^{m-1}}\le \frac{r_1}{m},$$
which gives $r_1\ge (1- \gamma)^\frac 1m r_2$, or  $|B_{r_1}|\ge (1- \gamma) |B_{r_2}|$.
Finally, one gets that $|A|\le \gamma |\Om|$.   Hence $\mathfrak{C}_{+}(\Omega,\gamma)\le \gamma |\Omega|$.

As a byproduct of the computation, we  observe that the constant $\gamma$ in Theorem \ref{the1} is sharp, and that equality holds for the ball. As soon as, $r_1< (1- \gamma)^\frac 1m r_2$, one gets that $v'(r_2)>0$. This means that as $v(r_2)=0$ the solution is not positive near the boundary of the ball.

In order to prove the converse inequality, let us  prove that for every $\varepsilon >0$, there exists a set $A \subset \Om$ of measure $\gamma |\Om|-\varepsilon$ such that
$v_{\Om, A, \gamma} \ge 0$. This will imply that $\mathfrak{C}_{+}(\Omega, \gamma)\ge \gamma |\Om|$.

The construction is based on the following observation. There exists a finite family of mutually disjoint balls $\cup_{i=1}^k B_i$ contained in $\Om$ such that
$$|\Om \setminus \cup_{i=1}^k B_i|< \varepsilon.$$
In every ball, we display the set $A_i$ of measure $\gamma |B_i|$ in an annulus centred at the center of $B_i$ and having $\partial B_i$ as external boundary.  Hence   $v_{B_i, A_i, \gamma} \ge 0$. Moreover,   since  the sets $B_i$ are mutually disjoint we get that
$$v_{\cup_iB_i, \cup_iA_i, \gamma} \ge 0.$$
We  have   the following.
\begin{lemma}\label{lem:ext}
Let $\Om_1 \sq  \Om_2 \sq \R^m$ be  open sets   with   finite measure, $f \in L^2( \Om_2),$ and  let   $u_1, u_2$ weak solutions of
$$ -\Delta u_i =f \mbox { on } \Om_i, u \in H^1_0(\Om_i), i=1,2.$$
If $u_1\ge 0$ on $\Om_1$ and  $f\ge 0$ on $\Om_2 \sm \Om_1$ then
$u_2\ge 0$ on $\Om_2$.
\end{lemma}
\begin{proof}
As a consequence of the hypotheses, {we get}
$$-\Delta u_1 \le  f \mbox{ in } {\mathcal D}'(\Om_2).$$
Hence, by the maximum principle
$$0\le u_1 \le u_2 \mbox { on }\Om_2.$$
\end{proof}

A direct consequence of Lemma \ref{lem:ext} is that if $\Om_1 \sq \Om_2$  then   $\mathfrak{C}_{+}(\Om_1, \gamma)\le \mathfrak{C}_{+}(\Om_2, \gamma)$. Indeed, for every measurable set $A \sq \Om_1$ such that $v_{\Om_1, A, \gamma} \ge 0$ we get $v_{\Om_2, A, \gamma} \ge 0$.

Coming back to the proof of Theorem \ref{the1}, using the additivity and monotonicity property of $\mathfrak{C}_{+}$ we get that
$$\mathfrak{C}_{+}(\Om, \gamma) \ge \gamma |\cup _i B_i|\ge \gamma |\Om| -\gamma\varepsilon.$$
  The theorem follows by letting $\varepsilon \rightarrow 0$.

\end{proof}

\section{Proof of Theorem \ref{the2}}\label{sec2}

We first introduce some basic notation and properties. For a non-empty open set $\Omega\subset \R^m$ we denote by $G_{\Om}(x,y),\,x\in \Om,\, y \in \Om,\,{ x\ne y,}$ the kernel of the resolvent of the Dirichlet Laplacian acting in $L^2(\Om)$. This function exists and is well defined for all $x\neq y$, provided $m\ge 3$. It also exists for $m=2$ for example under the hypothesis that the torsion function $v_{\Om}$ defined by approximation on balls, is locally finite.  The resolvent kernel is non-negative, symmetric in $x$ and $y$, and is monotone increasing in $\Om$. That is, if $\Om_1\subset \Om_2$, then
\begin{equation}\label{ee2.1}
0\le G_{\Om_1}(x,y)\le G_{\Om_2}(x,y),\, x\in \Om_1,\,y\in\Om_1, x\ne y.
\end{equation}
If $v_{\Om}$ is locally finite, then
\begin{equation*}
v_{\Omega}(x)=\int_{\Omega}dy\,G_{\Omega}(x,y).
\end{equation*}
The monotonicity in \eqref{ee2.1} implies that both the torsion function $v_{\Om}$, and torsion $T(\Om)$ are monotone increasing in $\Om$.

We have also that
\begin{align}\label{ee3}
v_{\Omega,A,\gamma}(x)&=\int_{\Om}dy\,G_{\Om}(x,y)\big(\gamma 1_{\Om\setminus A}(y)-(1-\gamma)1_A(y)\big)\nonumber \\
&=\int_{\Om}dy\,G_{\Om}(x,y)\big(\gamma 1_{\Om}(y)-1_A(y)\big)\nonumber \\ & =\gamma v_{\Om}(x)-\int_{A}dy\,G_{\Om}(x,y).
\end{align}
Formula \eqref{ee3} implies that
\begin{equation*}
-(1-\gamma)v_{\Omega} \le \,  v_{\Omega,A,\gamma}\le \gamma v_{\Omega}.
\end{equation*}

\noindent{\it Proof of Theorem \textup{\ref{the2}}. }
 Let $\Om=\Delta OAB$ be a triangle, with  $\alpha:=\angle BOA\le \frac{\pi}{3}$ at the origin, and oriented such that the positive $x$-axis is the bisectrix of that angle. Let $W_{\alpha}$ be the infinite wedge with vertex at $O$, and edges at angles $\pm\frac12\alpha$ with the positive $x$-axis, which contain the two sides $OA$ and $OB$ of $\Om$. Let $W_{\alpha, c}$ be the radial sector with area $c$ and edges at angles $\pm\frac12\alpha$. Then $W_{\alpha,c}\subset \Om$ for all $c$ sufficiently small.
We have by monotonicity that
\begin{align}\label{f1}
v_{\Om,W_{\alpha, c},\gamma}(x)&=\int_{\Om}dy\, G_{\Om}(x,y) \big(\gamma {\bf 1}_{\Om}-{\bf 1}_{W_{\alpha, c}}\big)(y)\nonumber \\&
\le \gamma \int_{W_{\alpha}}dy\,G_{W_\alpha}(x,y)- \int_{W_{\alpha,c}}dy\,G_{W_{\alpha,c}}(x,y)\nonumber \\ &
=\gamma v_{W_{\alpha}}(x)-v_{W_{\alpha,c}}(x).
\end{align}
In Cartesian coordinates $x=(x_1,x_2)$, we have that
\begin{equation}\label{f2}
v_{W_{\alpha}}(x_1,x_2)=\frac{x_2^2-s^2x_1^2}{2(s^2-1)},
\end{equation}
where $s=\tan(\alpha/2)$.
In polar coordinates $x=(r;\theta)$ we have by p.279 in \cite{TG} for the sector with radius $a=(2c/{\alpha})^{1/2}${,}
\begin{align}\label{f3}
v_{W_{\alpha,c}}(r;\theta)=&\frac{r^2}{4}\bigg(\frac{\cos(2\theta)}{\cos \alpha}-1\bigg)\nonumber \\ & +\frac{4a^2\alpha^2}{\pi^3}\sum_{n=1,3,5,...}\frac{(-1)^{(n+1)/2}(r/a)^{n\pi/{\alpha}}\cos(n\pi\theta/{\alpha})}
{n\bigg(n+\frac{2\alpha}{\pi}\bigg)\bigg(n-\frac{2\alpha}{\pi}\bigg)}.
\end{align}
We observe that for $\theta=0$ the terms in the series in the right-hand side of \eqref{f3} are alternating and decreasing in absolute value.
Hence
\begin{equation*}
v_{W_{\alpha,c}}(r;0)=\frac{r^2}{4}\bigg(\frac{1}{\cos \alpha}-1\bigg)-\frac{4a^2\alpha^2}{\pi^3}\bigg(\frac{r}{a}\bigg)^{\pi/{\alpha}}\bigg(1-\frac{4\alpha^2}{\pi^2}\bigg)^{-1}.
\end{equation*}
By \eqref{f2} $v_{W_{\alpha}}(x_1,0)=\frac{s^2x_1^2}{2(1-s^2)}$, and so in polar coordinates,
\begin{equation}\label{f5}
v_{W_{\alpha}}(r;0)=\frac{r^2}{4}\bigg(\frac{1}{\cos \alpha}-1\bigg).
\end{equation}
{By \eqref{f1}-\eqref{f5} we have}
\begin{equation*}
v_{\Om,W_{\alpha, c},\gamma}(x_1,0)\le (\gamma-1)\frac{s^2x_1^2}{2(1-s^2)}+O\big(x_1^{\pi/{\alpha}}\big),\, x_1\downarrow 0,
\end{equation*}
which is negative for all $x_1$ sufficiently small.

We see from the proof above that we could have chosen any angle of the triangle provided that angle is strictly less than $\pi/2$.
The proof above also shows that the infinite wedge $W_{\alpha},\alpha<\pi/2$ with radial sector $W_{\alpha,c},c>0$ has a sign changing solution
$v_{W_{\alpha},W_{\alpha,c},\gamma}$.
\hspace*{\fill }$\square $

\section{Proof of Theorem \ref{the3}\label{sec3}}

 \noindent{\it Proof of Theorem \textup{\ref{the3}}.}   Let us start by observing that the following covering property holds: for every $x\in \Om$, there exists a ball $B$ of radius $R$ such that $x \in B\subset \Om$. Indeed, let   $\bar x \in \partial \Om$ be a point which realises the distance to the boundary. Since the boundary of $\Om$ is of class $C^2$, then $x-\bar x$ is normal to the boundary $\partial \Om$ at $\bar x$. If $|x-\bar x|\ge R$, then $B_R(x) \subset \Om$. If $|x_0-\bar x|<  R$, then $x$ belongs to the ball of radius $R$ tangent to $\partial \Om$ at $\bar x$.

Assume for a contradiction that
$$\mathfrak{C}_{-}(\Omega,\gamma)< \mathfrak{C}_{-}(B_R, \gamma).$$
For every $\vps >0$ such that $\mathfrak{C}_{-}(\Omega,\gamma)+\vps< \mathfrak{C}_{-}(B_R, \gamma)$, there exists a set $A_\vps\subset \Om$ such that $|A_\vps | \le \mathfrak{C}_{-}(\Omega,\gamma)+\vps$ and
$${\rm essinf} \; v_{\Om, A_\vps,\gamma}<0,$$
the infimum being attained  at  $x_\vps$.
Taking a sequence $\vps \rightarrow 0$, we may assume (up to extracting suitable subsequences) that
$$1_{A_\vps} \rightarrow g \mbox{ weakly-$\star$ in } L^\infty,  \;\; x_\vps  \rightarrow x_* \in \overline \Om.$$
Then $ \int_\Om g   = \mathfrak{C}_{-}(\Omega,\gamma)$.  Let $v_{\Om,g,\gamma}$ denote  the solution of
\begin{equation*}
-\Delta v=\gamma(1-g)-(1-\gamma)g,\, v\in H_0^1(\Omega),
\end{equation*}
{we get} $$v_{\Om, A_\vps,\gamma} \rightarrow v_{\Om,g,\gamma}$$
uniformly on $\overline \Omega$. This is a consequence of the elliptic regularity of the solutions, which are uniformly bounded in $C^{1, \alpha} (\overline \Omega)$ for some $\alpha >0$. Consequently, $ v_{\Om,g,\gamma} \ge 0$ in $\Om$. Indeed, for $x^*$ a minimum point of $ v_{\Om,g,\gamma}$ with
$ v_{\Om,g,\gamma}(x^*) <0$, we can modify $g$ slightly to find  a new function $\tilde g$, such that
$$0 \le \tilde g \le 1, \;\; \int_\Om \tilde g  < \mathfrak{C}_{-}(\Omega,\gamma), \;\;  v_{\Om,\tilde g,\gamma}(x^*) <0. $$
From the density of the characteristic functions, we can find a sequence of sets $\tilde A_\delta$ such that
$1_{\tilde A_\delta} \rightarrow \tilde g $ weakly-$\star$ in $L^\infty$, and $|\tilde A_\delta| =  \int_\Om \tilde g $. In particular, $v_{\Om,\tilde A_\delta,\gamma}(x^*) <0$. This contradicts the definition of $\mathfrak{C}_{-}(\Omega,\gamma)$.

Consequently, $v_{\Om,g,\gamma}(x^*) =0$. There are two possibilities: either $x^* \in \Om$, or $x^* \in \partial \Om$. Assume first that $x^* \in \Om$. As a consequence of the covering property, there exists a ball $B$ of radius $R$ such that $x_0 \in B\subset \Om$. In particular, this implies that
$v_{\Om,g,\gamma} \ge 0$ on $\partial B$. The maximum principle {gives}
$$v_{\Om,g,\gamma} \ge v_{B,g,\gamma}, \mbox{ on } B.$$
Consequently $v_{B,g,\gamma}(x^*) \le 0$. Clearly
\begin{equation}\label{e2.4}
\int_B g  \le \mathfrak{C}_{-}(\Omega,\gamma)< \mathfrak{C}_{-}(B_R, \gamma).
\end{equation}

{\bf Case 1.} In case $v_{B,g,\gamma}(x^*) < 0$, we immediately get a contradiction since, as above, we can build a sequence of sets $\tilde A_\delta\subset B$ such that
$1_{\tilde A_\delta} \rightarrow \tilde g\cdot 1_B $ weakly-$\star$ in $L^\infty (B)$, and $|\tilde A_\delta| =  \int_B  g $. By the uniform convergence, we get that
 $v_{B,\tilde A_\delta,\gamma}(x^*) <0$, so that $\mathfrak{C}_{-}(B_R, \gamma) \le  \int_B g $. This contradicts  \eqref{e2.4}.
 \smallskip

{\bf Case 2.} In case $v_{B,g,\gamma}(x^*) =0$, we claim that either $g$ is itself a characteristic function, or we can find another function $\tilde g$ such that
 $$0\le \tilde g \le 1,  \int_B \tilde g < \mathfrak{C}_{-}(B_R, \gamma), \mbox{ and } v_{B,\tilde g,\gamma}(x^*) <0.$$
 Assume that $g$ is a characteristic function. Then $g=1_A$. Taking a new set $A\subset \tilde A \subset B$, such that $|A|< |\tilde A| < \mathfrak{C}_{-}(B_R, \gamma)$ we get by the maximum principle that  $v_{B,\tilde A,\gamma}(x^*) <0$, in contradiction with the definition of  $\mathfrak{C}_{-}(B_R, \gamma)$.

Assume that $g$ is not a characteristic function on $B$. Then, for some value $\delta >0$ the set $U_\delta= \{ x\in B : \delta \le g(x) \le 1-\delta \}$ has positive Lebesgue measure. We put $\tilde g = g+ s 1_{U_\delta}$, where $s>0$ is small enough such that
$ \int_B   \tilde g  <\mathfrak{C}_{-}(B_R, \gamma)$. By the maximum principle, we get $v_{B,\tilde g,\gamma}(x^*) <0$. In this case, we are back to Case 1.

Assume now that $x^* \in \partial \Om$. Let ${\bf n_{x^*}}$ be the outward normal vector at $x^*$. Let $\overline x_\vps$  be  the projection on $\partial \Om$ of $x_\vps$. Since $\Om$ is of class $C^2$, we get $\overline x_\vps \rightarrow x^*$ and that there exists a point $y_\vps$ on the segment $[\overline x_\vps, x_\vps]$ such that $\nabla v_{\Om,A_\vps,\gamma}(y_\vps) \cdot {\bf n_{\overline x_\vps}} \ge 0$. Passing to the limit, we get
$$\nabla v_{\Om,g,\gamma}(x^*) \cdot {\bf n_{x^*}} \ge 0.$$
Meanwhile, $x^*$ is a minimum point of $v_{\Om,g,\gamma},$ so that
$$\frac{\partial v_{\Om,g,\gamma}}{\partial {n}}(x^*) \le 0.$$
Hence,
$$\frac{\partial v_{\Om,g,\gamma}}{\partial {n}}(x^*) =0.$$
  Using the $R$-smoothness at $x^*$, the ball $B\subset \Om$ of radius $R$ tangent  to  $\partial \Om$ at $x^*$ stays in $\Om$.  Since $
 v_{\Om,g,\gamma}\ge 0$, by the maximum principle we get $v_{\Om,g,\gamma}\ge  v_{B,g,\gamma}.$  By the Hopf maximum  principle, applied to $v_{\Om,g,\gamma}-  v_{B,g,\gamma}$ on $B$ at the minimum point $x^*\in \partial B$,  we have either that
 $$\frac{\partial v_{\Om,g,\gamma}}{\partial {n}}(x^*)  - \frac{\partial v_{B,g,\gamma}}{\partial {n}}(x^*) < 0,$$
 or that $v_{\Om,g,\gamma}-  v_{B,g,\gamma}=0$ on $B$.
In the first situation,
 $$\frac{\partial v_{B,g,\gamma}}{\partial {n}}(x^*) < 0,$$ which means that $ v_{B,g,\gamma}$ takes negative values close to $x^*$. Then, we conclude as in Case 1, above. In the second situation, if we find a point $\overline x \in \partial B \cap \Om$, we can conclude as in Case 2 since  $v_{B,g,\gamma}(\overline x) =0$. The alternative is that $\partial B \subset \partial \Om$ so that $\Om=B$, and we have a contradiction.

To prove \eqref{e00} we let $m\ge 3$, and let $H$ be an open half-space. Then
\begin{equation}\label{e6}
G_{H}(x,y)=c_m\bigg(|x-y|^{2-m}-|x^*-y|^{2-m}\bigg),
\end{equation}
where $x^*$ is the reflection of $x$ with respect to $\partial H$, and
\begin{equation*}
c_m=\frac{\Gamma((m-2)/2)}{4\pi^{m/2}}.
\end{equation*}
By \eqref{ee3}, and monotonicity we have that
\begin{align}\label{e9}
v_{B_R,A,\gamma}(x) \ge \gamma v_{B_R}(x)-\int_{A}dy\,G_{H_{\bar x}}(x,y),
\end{align}
where $H_{\bar x}$ is the half-space tangent to $B_R$ at $\bar x\in \partial B_R$.
Note that $|x^*-\bar x|=|\bar x-x|$. Moreover, $|x-y|\le |x^*-y|,\, y\in \Omega$. Hence,
\begin{equation}\label{e10}
0\le |x-y|^{2-m}-|x^*-y|^{2-m}\le (m-2)|x-x^*||x-y|^{1-m}.
\end{equation}
Let
\begin{equation*}
A_x^*=\{y:|y-x|<r_A\},
\end{equation*}
where
\begin{equation}\label{e12}
\omega_mr_A^m=|A|.
\end{equation}
By \eqref{e6}, \eqref{e9}, \eqref{e10}, and radial rearrangement of $A$ about $x$, we {have}
\begin{align}\label{e13}
\int_{A}dy\,G_{B_R}(x,y)&\le (m-2)c_m|x-x^*|\int_{A}dy\,|x-y|^{1-m}\nonumber \\ &
\le (m-2)c_m|x-x^*|\int_{A_x^*}dy\,|x-y|^{1-m}\nonumber \\ &=(m-2)mc_m\omega_mr_A|x-x^*|\nonumber \\ &
=2r_A|x-\bar x|.
\end{align}

The following will be used in the proof of \eqref{e00}, and in the proof of \eqref{ef9} and \eqref{ef10} in Remark 6 below.
\begin{lemma}\label{lem1}
If $\Omega$ is an open set in $\R^m, m\ge 2$ with $R$-smooth boundary, and if
if $\lambda(\Omega)>0$, then
\begin{equation}\label{e14}
v_{\Omega}(x)\ge \frac{|x-\bar x|R}{2m}.
\end{equation}
\end{lemma}
\begin{proof}
Recall that
\begin{equation*}
v_{  B_r(c) }(x)=\frac{r^2-|x-c|^2}{2m}.
\end{equation*}
We first consider the case $|x-\bar x|>R$. Then, by domain monotonicity of the torsion function, and \eqref{e13}
\begin{equation*}
v_{\Omega}(x)\ge v_{ B_{|x-\bar x|}(x) }(x)=\frac{|x-\bar x|^2}{2m}\ge \frac{|x-\bar x|R}{2m}.
\end{equation*}
We next consider the case $|x-\bar x|\le R.$ Since $\partial\Omega$ is $R$-smooth, there exists  $  B_R(c_x) \subset \Omega$ such that $|c_x-\bar x|=R$.
Hence, by \eqref{e14},
\begin{equation*}
v_{\Omega}(x)\ge v_{  B_R(c_x) }(x)=\frac{R^2-|x-c_x|^2}{2m}\ge \frac{(R-|x-c_x|)R}{2m}=\frac{|x-\bar x|R}{2m}.
\end{equation*}
In either case we conclude \eqref{e14}.
\end{proof}
By \eqref{e13} and \eqref{e14} we have that
\begin{equation}\label{e17}
v_{B_R,A,\gamma}(x)\ge\gamma \frac{|x-\bar x|R}{2m}-2r_A|x-\bar x|.
\end{equation}
The right-hand side of \eqref{e17} is non-negative for $r_A\le \gamma R/(4m)$. This is, by \eqref{e12}, equivalent to \eqref{e00}.

Consider the case $m=2$. Then
\begin{equation*}
G_{H}(x,y)=\frac{1}{2\pi}\log\bigg(\frac{|x^*-y|}{|x-y|}\bigg).
\end{equation*}
By the triangle inequality,
\begin{equation*}
G_{B_R}(x,y)\le G_{H_{\bar x}}(x,y)\le\frac{1}{2\pi}\log\bigg(\frac{|x^*-x|+|x-y|}{|x-y|}\bigg)\le \frac{|x^*-x|}{2\pi|x-y|} .
\end{equation*}
Hence we have that
\begin{equation*}
\int_{A}dy\,G_{B_R}(x,y)\le (2\pi)^{-1}|x-x^*|\int_{A_x^*}dy\,|x-y|^{-1}=2r_A|x-\bar x|.
\end{equation*}
The remaining arguments follow those of the case $m\ge 3$, as the right-hand side above equals the right-hand side of \eqref{e13}.

To prove \eqref{ex1} we let $m\ge 3$. By scaling it suffices to prove \eqref{ex1} for $R=1$. Let $a\in (0,1)$. We obtain an upper bound for $a$ such that
$v_{B_1,B_a,\gamma}(0)<0.$ Note that
\begin{equation}\label{ex3}
G_{B_1}(0,y)=\frac{\Gamma((m-2)/2)}{4\pi^{m/2}}\big(|y|^{2-m}-1\big).
\end{equation}
Hence, by \eqref{ex3} we have that
\begin{align}\label{ex4}
v_{B_1,B_a,\gamma}(0)&=\gamma v_{B_1}(0)-\int_{B_a}dy\,G_{B_1}(0,y)\nonumber \\ &
=\frac{\gamma}{2m}-\frac{\Gamma((m-2)/2)}{4\pi^{m/2}}m\omega_m\int_{[0,a]}dr\big(r-r^{m-1}\big)\nonumber \\ &
\le \frac{\gamma}{2m}-\frac{a^2}{2m}.
\end{align}
The right-hand side of \eqref{ex4} is negative for $a>\gamma^{1/2}$. This implies \eqref{ex1}.

To prove \eqref{ex2} we let $m=2,\, a\in (0,1)$, and note that
\begin{equation*}
G_{B_1}(0,y)=-\frac{1}{2\pi}\log |y|.
\end{equation*}
Hence,
\begin{align}\label{ex6}
v_{B_1,B_a,\gamma}(0)&=\frac{\gamma}{4}+\int_{[0,a]}dr\,r\log r \nonumber \\ &
=\frac{\gamma}{4}-\frac{a^2}{4}+\frac{a^2}{4}\log a^2.
\end{align}
Let
\begin{equation}\label{ex7}
a=\Big(1+\log \Big(\frac{1}{\gamma}\Big)\Big)^{-1/2} \gamma^{1/2}.
\end{equation}
Then $a\in (0,1)$, and by \eqref{ex6} and \eqref{ex7},
\begin{equation*}
 v_{B_1,B_a,\gamma}(0)\le -\frac{\gamma}{4}\frac{\log\Big(1+\log\Big(\frac{1}{\gamma}\Big)\Big)}{1+\log\Big(\frac{1}{\gamma}\Big)}<0.
\end{equation*}
This implies \eqref{ex2}.
\hspace*{\fill }$\square $

\section{Proof of Theorem \ref{the4} \label{sec4}}

{\it Proof of Theorem \textup{\ref{the4}}.}
The proof of Theorem \ref{the4} relies on some basic facts on the connection between torsion function, Green function, and heat kernel. These have been exploited elsewhere in the literature. See for example \cite{MvdBB}.
We recall that (see \cite{EBD3}, \cite{GB}, \cite{GB1}) the heat equation
\begin{equation*}
  \Delta u=\frac{\partial u}{\partial t}\,  \textup{on}\,\Om\times \R^+,
\end{equation*}
has a unique, minimal, positive fundamental solution $p_{\Omega}(x,y;t),$
where $x\in \Omega$, $y\in \Omega$, $t>0$. This solution, the
heat kernel for $\Omega$, is symmetric in $x,y$, strictly positive,
jointly smooth in $x,y\in \Omega$ and $t>0$, and it satisfies the
semigroup property
\begin{equation*}
p_{\Omega}(x,y;s+t)=\int_{\Omega}dz\ p_{\Omega}(x,z;s)p_{\Omega}(z,y;t),
\end{equation*}
for all $x,y\in {\Omega}$ and $t,s>0$. If $\Omega$ is an open subset of $\R^m$, then, by minimality,
\begin{equation}\label{e34}
p_{\Omega}(x,y;t)\le p_{\R^m}(x,y;t)=(4\pi t)^{-m/2}e^{-|x-y|^2/(4t)},\, x\in \Omega,\ y\in  \Omega, \ t>0.
\end{equation}
It is a standard fact that for $\Omega$ open in $\R^m$,
\begin{equation}\label{e35}
G_{\Omega}(x,y)=\int_{[0,\infty)}dt\,p_{\Omega}(x,y;t),
\end{equation}
whenever the integral with respect to $t$ converges.
We {have}
\begin{equation*}
    v_{\Omega}(x)=\int_{[0,\infty)}dt\,\int_{\Omega}dy\, p_{\Omega}(x,y;t).
\end{equation*}
By the heat semigroup property, we have that for $x\in \Omega,y\in \Omega, t>0,$
\begin{align}\label{e37}
 p_{\Omega}(x,y;t)&=\int_{\Omega}dr\, p_{\Omega}(x,r;t/2)p_{\Omega}(r,y;t/2)\nonumber \\ &
 \le\bigg(\int_{\Omega}dr\, \big(p_{\Omega}(x,r;t/2)\big)^2\Bigg)^{1/2}\bigg(\int_{\Omega}dr\, \big(p_{\Omega}(r,y;t/2)\big)^2\Bigg)^{1/2}\nonumber \\&
 =\big(p_{\Omega}(x,x;t)p_{\Omega}(y,y;t)\big)^{1/2}.
\end{align}
Furthermore, for all $s\in (0,t),$
\begin{equation}\label{e38}
 p_{\Omega}(z,z;t)\le e^{-s\lambda(\Omega)} p_{\Omega}(z,z;t-s).
\end{equation}
So choosing $s=t/2$ in \eqref{e38}, and subsequently using \eqref{e37} gives that
\begin{equation}\label{e39}
 p_{\Omega}(x,y;t)\le e^{-t\lambda(\Omega)/3}\big(p_{\Omega}(x,x;t/2)p_{\Omega}(y,y;t/2)\big)^{1/3}p_{\Om}(x,y;t)^{1/3}.
\end{equation}
By \eqref{e34}, both diagonal heat kernels in the right-hand side of \eqref{e39} are bounded
by $(2\pi t)^{-m/2}$, and $p_{\Om}(x,y;t)^{1/3}\le (4\pi t)^{-m/6}e^{-|x-y|^2/(12t)}.$ Hence by \eqref{e39},
\begin{align}\label{e40}
&p_{\Omega}(x,y;t)\le 2^{m/3}(4\pi t)^{-m/2}e^{-t\lambda(\Omega)/3-|x-y|^2/(12t)}\nonumber \\ &
\le 2^{m/3}\sup_{t>0}\big(e^{-t\lambda(\Om)/6-|x-y|^2/(24t)}\big)e^{-t\lambda(\Om)/6}(4\pi t)^{-m/2}e^{-|x-y|^2/(24t)}\nonumber \\ &
=2^{m/3}e^{-t\lambda(\Om)/6-|x-y|\lambda(\Omega)^{1/2}/6}(4\pi t)^{-m/2}e^{-|x-y|^2/(24t)}.
\end{align}

Let $c>0$, and let ${r_1}$ be the radius of a
ball of volume $\frac{c}{2}$ and ${r_2} =2^\frac 1m r_1$ be the radius of a ball of volume $c$. Following the result of Lieb \cite[Theorem 1]{Li83}, there exists a translation $x$ of  $B_{r_1}$  such that
$$\lb(\Om) + \lb (B_{r_1})\ge \lb(\Om\cap B_{r_1}(x)).$$

The Kohler-Jobin inequality asserts that  (see for instance \cite{Br14}) there exists $C_2(m)>0$ such that for every open set $\Om$,
\begin{equation}\label{ea1}
\lb(\Om) T (\Om)^\frac{2}{m+2} \ge C_2(m).
\end{equation}
This, together with the Lieb inequality, implies
\begin{equation}\label{ea3}T(A')=\int_{ \Om\cap B_{r_1}(x)} v_{ \Om\cap B_{r_1}( x )} \ge \bigg(\frac{C_2(m)}{\lambda(\Omega\cap B_{r_1}(x))}\bigg)^\frac{m+2}{2}\ge \bigg ( \frac{C_2(m)}{ \lb(\Om) + \lb (B_{r_1})} \bigg )^\frac{m+2}{2},
\end{equation}
where $A' = B_{r_1}(x) \cap \Om$. We put $A=B_{r_2}(x)\cap\Om$.

We estimate the  integral  of $v_{\Om,A,\gamma} $ on the set $A'$ as follows:
\begin{align}\label{ea0}
\int_{A'} v_{\Om, A,\gamma} &= \int_{A' }\,dx \bigg ( \int _\Om\,dy\, G_\Om(x,y) (\gamma{\bf 1}_{\Om \sm A}(y) -(1-\gamma){\bf 1}_A(y))\bigg) \nonumber \\ &
 =\gamma\int_{A' }\,dx  \int _{\Om \sm A}dy\, G_\Om(x,y) -  (1-\gamma)\int_{A' }\,dx  \int _{A }dy\, G_\Om(x,y).
\end{align}
By monotonicity, we have that
\begin{equation}\label{ea4}
(1-\gamma)\int_{A' }dx\,  \int _{A }dy\, G_\Om(x,y)1_A(y) \ge (1-\gamma)\int_{A' } v_{A'}= (1-\gamma)T(A').
\end{equation}
For all $x\in A'$, and for all $y\in \Om \sm A$ we have that
$$|x-y|\ge \frac {1}{2m}\big (c/\omega_m\big)^{1/m} .$$
By \eqref{e35} and \eqref{e40}, and the preceding inequality,
\begin{align}\label{ea8}
 &\gamma\int_{A'}\, dx  \int _{\Om \sm A} dy\, G_{\Om}(x,y)\nonumber \\ & \le\gamma2^{m/3}6^{m/2}\int_{[0,\infty)}dt\int_{A' }\, dx  \int _{\Om \sm A} dy\,e^{-t\lambda(\Om)/6-|x-y|\lambda(\Omega)^{1/2}/6}p_{\R^m}(x,y;6t)\nonumber \\ &\le \gamma 2^{m/3}6^{m/2}e^{-(c/\omega_m)^{1/m}\lambda(\Omega)^{1/2}/(12m)}\nonumber \\ &\hspace{30mm}\times\int_{[0,\infty)}dte^{-t\lambda(\Om)/6}\int_{A' }\, dx  \int _{\Om \sm A} dy\,p_{\R^m}(x,y;6t)\nonumber \\ &\le \gamma 2^{m/3}6^{m/2}e^{-(c/\omega_m)^{1/m}\lambda(\Omega)^{1/2}/(12m)}\nonumber\\&\hspace{30mm}\times\int_{[0,\infty)}dte^{-t\lambda(\Om)/6}\int_{A'}\, dx  \int _{\R^m} dy\,p_{\R^m}(x,y;6t)\nonumber \\ &=\gamma 2^{m/3}6^{(m+2)/2}e^{-(c/\omega_m)^{1/m}\lambda(\Omega)^{1/2}/(12m)}|A'|\lambda(\Om)^{-1}\nonumber \\ &\le \gamma 2^{5m/6}3^{(m+2)/2}e^{-(c/\omega_m)^{1/m}\lambda(\Omega)^{1/2}/(12m)}c\lambda(\Om)^{-1}.
\end{align}
By \eqref{ea3}, \eqref{ea4}, \eqref{ea0}, and \eqref{ea8}, we find
\begin{align}\label{ea2}
\int_{A'} v_{\Om,A,\gamma}
\le  \gamma 2^{5m/6}3^{(m+2)/2}&e^{-(c/\omega_m)^{1/m}\lambda(\Omega)^{1/2}/(12m)}c\lambda(\Om)^{-1}\nonumber \\ & - (1-\gamma)\Big ( \frac{C_2(m)}{ \lb(\Om) + \lb (B_{r_1})} \Big )^\frac{m+2}{2}.
\end{align}
In order to bound the right-hand side of \eqref{ea2} from above, we have
\begin{align}\label{ea5}
&\Big ( \frac{C_2(m)}{ \lb(\Om) + \lb (B_{r_1})} \Big )^\frac{m+2}{2}\ge \bigg(\frac{C_2(m)^{1/2}}{\lb(\Om)^{1/2}+\lb (B_{r_1})^{1/2}}\bigg)^{m+2}\nonumber \\ & \hspace {25mm}
=\bigg(\frac{c}{\omega_m}\bigg)^{(m+2)/m}\bigg(\frac{C_2(m)^{1/2}}{\lb(\Om)^{1/2}(c/\omega_m)^{1/m}+2^{1/m}\lb (B_{1})^{1/2}}\bigg)^{m+2},
\end{align}
where we have used the scaling $\lb (B_{r_1})=r_1^{-2}\lb (B_{1})$.

In order to bound the first term in the right-hand side of \eqref{ea5} from above we use the inequality $e^{-x}\le \frac{((m+2)/e)^{m+2}}{x^{m+2}},\, x> 0$. We have
\begin{align}\label{ea6}
&e^{-(c/\omega_m)^{1/m}\lambda(\Omega)^{1/2}/(12m)}\nonumber \\ &=e^{2^{1/m}\lb (B_{1})^{1/2}/(12m)}e^{-((c/\omega_m)^{1/m}\lambda(\Omega)^{1/2}+2^{1/m}\lb (B_{1})^{1/2})/(12m)}\nonumber \\ &
\le e^{2^{1/m}\lb (B_{1})^{1/2}/(12m)}\nonumber \\ &\hspace{10mm}\times(12m(m+2)/e)^{m+2}((c/\omega_m)^{1/m}\lambda(\Omega)^{1/2}+2^{1/m}\lb (B_{1})^{1/2})^{-(m+2)}.
\end{align}
By \eqref{ea5} and \eqref{ea6}, we obtain that the right-hand side of \eqref{ea2} is bounded from above by $0$, provided
\begin{align*}
c\ge C_1(m)\bigg(\frac{\gamma}{1-\gamma}\bigg)^{m/2}\lambda(\Om)^{-m/2},
\end{align*}
with $C_1(m)$ given by \eqref{ee11.0}.  This implies the bound for $C_{-}(\Om,\gamma)$ in \eqref{ee12}.

\hspace*{\fill }$\square $

\section{Proof of Theorem \ref{the4.1} \label{sec4.1}}
We start with the following.
\begin{lemma}\label{lem4.1a}
There exists $\vps=\vps(m, \gamma) >0$ such that for every open set $\Om \sq \R^m$ with finite torsion and for every $x_0 \in \Om$ the following holds
$$ \mbox{if } v_\Om (x) \le \vps \mbox { for a.e. } x \in B_1(x_0) \mbox{ then } v_{\Om, B_1(x_0), \gamma} (x_0) \le 0.$$
\end{lemma}
  Note that a consequence of the lemma above,   for every $\delta >0$
\begin{equation}\label{eq4.1a}
\mbox{if } v_\Om (x) \le \vps \mbox { for a.e. } x \in B_{1+\delta}(x_0) \mbox{ then } v_{\Om, B_{1+\delta}(x_0), \gamma} \le 0 \mbox{ on } B_\delta(x_0).
\end{equation}
\begin{proof}
Assume  for the moment that $\Om$ is  bounded and smooth. Let $x_0 \in \Om$ such that
$$v_\Om (x) \le \vps \mbox { for a.e. } x \in B_1(x_0),$$ for some value $\vps >0$ that will be specified later in the proof.
We observe that $ v_{\Om, B_1(x_0), \gamma}$ is Lipschitz so that for every $r \in (0,1)$ one can define
$$M(r):= \sup_{x \in \partial B_r(x_0)}  v_{\Om, B_1(x_0), \gamma}(x).$$
The function $ M:(0,1)\to \R$ is Lipschitz and bounded from above by $\vps$. If there exists some $r \in (0,1)$ such that $M(r)=0$ then the assertion of the theorem is proved since one gets by the maximum principle that $v_{\Om, B_1(x_0), \gamma } \le 0$ on $B_r(x_0)$. So we can assume that $M >0$ on $(0,1)$. Then, the   supremum above is achieved at  a point $x_r \in \partial B_r(x_0) \cap \Om$.

Moreover,
$$M''(r)+\frac{m-1}{r} M'(r) \ge 1-\gamma,$$
in the viscosity sense on $(0, 1)$.  For every $0<\vps<R\le 1$ we introduce the equation
$$\phi_{\vps, R}''(r)+\frac{m-1}{r} \phi_{\vps, R}'(r) =  1-\gamma, \mbox{ \rm on } (\vps,R), \; \phi_{\vps, R}(\vps)=M(\vps), \phi_{\vps, R}(R)=M(R).$$
By the comparison principle (see for instance \cite[Theorem 1.1]{Tr88}) we get that $M\le \phi_{\vps, R}$ on $(R, d)$. In particular, this implies that $\phi$ is   non-negative. If $M$ is differentiable at $R$, then  $\phi_{\vps, R}'(R) \le M'(R)$.

  Multiplying the equation for $\phi_{\vps, R}$ by $r^{m-1}$ and integrating between   $r$ and $R$ gives
$$R^{m-1} \phi_{\vps, R}'(R) -r^{m-1} \phi_{\vps, R}'(r)= (1-\gamma) \Big ( \frac{R^m}{m} - \frac{r^m}{m}\big).$$
  Dividing by $r^{m-1}$ and integrating over $(\vps, R)$ yields
\begin{align*}R^{m-1} \phi_{\vps, R}'(R)& \int_\vps ^R \frac{1}{r^{m-1}} dr -(M(R)-M(\vps))\nonumber \\ &
=(1-\gamma) \frac{R^m}{m}  \int_\vps ^R \frac{1}{r^{m-1}} dr  -\frac{1-\gamma}{2m} (R^2-\vps ^2).
\end{align*}
Since $M$ is Lipschitz and $\lim_{\vps \to 0} \int_\vps ^R \frac{1}{r^{m-1}} dr=+\infty$, we get
$$\lim_{\vps \to 0}\phi_{\vps, R}'(R) = (1-\gamma) \frac Rm.$$
Finally,
$$M'(R) \ge (1-\gamma) \frac Rm.$$
  Integrating over $(0,1)$ gives
$$M(1)-M(0) \ge \frac{1-\gamma}{2m}.$$
  Since   $M \ge 0$,
$$M(1) \ge \frac{1-\gamma}{2m}.$$
  Taking into account that $M \le \gamma v_\Om$, and putting
$$\vps :=  \frac{1-\gamma}{2m\gamma},$$
concludes   the proof.

Assume now that $\Om$ is open and with finite torsion. Assume that $x_0 \in \Om$ is such that
$$ v_\Om (x) \le \vps \mbox { for a.e. } x \in B_1(x_0).$$
Let $(\Om_n)_n$ be an increasing sequence of open, smooth sets such that $\Om=\cup_n \Om_n$.   For all $n$ sufficiently large,   $x_0\in \Om_n$. Moreover, by the maximum principle,
$$v_{\Om_n} (x) \le \vps \mbox { for a.e. } x \in B_1(x_0).$$
Then $v_{\Om_n, B_1(x_0), \gamma} (x_0)\le 0$.   At   the same time, $v_{\Om_n, B_1(x_0), \gamma}$ converges to $v_{\Om, B_1(x_0), \gamma}$ uniformly on any compact contained in $\Om\cap B_1(x_0)$.   Hence
 $v_{\Om_n, B_1(x_0), \gamma}(x_0)\le 0$.
\end{proof}
\noindent{\it Proof of Theorem \textup{\ref{the4.1}}.} Let $\Om$ be open, connected and with finite torsion. If
$\mathfrak{C}_{-}(\Omega,\gamma)=0$, then inequality \eqref{ee11.1} is satisfied. Assume
$\mathfrak{C}_{-}(\Omega,\gamma)>0$. Then, for every $\delta >0$, there exists
$t>0$ such that
\begin{equation}\label{eee1}
\mathfrak{C}_{-}(t\Omega,\gamma)=(1+ \delta) |B_1|.
\end{equation}
  By \eqref{ee9},
\begin{equation}\label{eq4.1f}
t= \Big (\frac{(1+\delta)|B_1|}{\mathfrak{C}_{-}(\Omega,\gamma)}\Big)^\frac 1m.
\end{equation}
If there exists $x_0 \in t\Om$ such that
$v_{t\Om} \le \vps$ on $B_1(x_0)$, then by Lemma \ref{lem4.1a} we get
$v_{\Om, B_1(x_0), \gamma} (x_0)\le 0$, so that $\mathfrak{C}_{-}(t\Omega,\gamma)\le  |B_1|$, in contradiction with our choice. Consequently, for every $x _0 \in t\Om$, $\sup_{B_1(x_0)}  v_{t\Om}>\frac{1-\gamma}{2m\gamma}$. This inequality leads to a relationship between $T(t\Om)$ and $\textup{diam} (t\Om)$.

Indeed, if for some $y \in t\Om$, $v_{t\Om} (y) > \frac{1-\gamma}{2m\gamma}$, then for every $r >0$
$$\int_{B_r(y)} v_{t\Om}(x) dx \ge r^m |B_1| \Big (\frac{1-\gamma}{2m\gamma} -  \frac{r^2}{2(m+2)}\Big ).$$
This follows from the fact that
$x\mapsto v_{t\Om}(x) + \frac{|x-y|^2}{2m}$ is subharmonic on $\R^m$.   We have extended $v_{\Om}$ to all of $\R^m$ by putting $v_{\Om}(x)=0$ on $\R^m\setminus \Om.$

Choosing $r$ such that
\begin{equation}\label{eq4.10}
\frac{r^2}{2(m+2)}= \frac{1-\gamma}{4m\gamma},
\end{equation}
we get
$$\int_{B_r(y)} v_{t\Om}(x) dx \ge \frac{(m+2)^{m/2}}{2^{(m+4)/2}m^{(m+2)/2}} \Big (\frac{1-\gamma}{\gamma}\Big )^{(m+2)/2}|B_1|.$$
Assume that $N$ is an integer such that
$$N(2r+2) \le \textup{diam} (t\Om) \le (N+1)(2r+2).$$
Then,
$$T(t\Om) \ge N \frac{(m+2)^{m/2}}{2^{(m+4)/2}m^{(m+2)/2}} \Big (\frac{1-\gamma}{\gamma}\Big )^{(m+2)/2}|B_1|.$$
  If $N\ge 1$, then   using the inequality $N +1 \le 2N$ we get
\begin{equation}\label{eq4.1b}
\textup{diam}(t \Om)\le 2(2r+2) \Big (  \frac{(m+2)^{m/2}}{2^{(m+4)/2}m^{(m+2)/2}} \Big (\frac{1-\gamma}{\gamma}\Big )^{(m+2)/2}|B_1|\Big)^{-1} T(t\Om)
  \end{equation}
  If $N=0$, then  we  observe that $\textup{diam} (t\Om) \le 2r+2$.   Inequality   \eqref{ee12} (which follows from Theorem \ref{the4}) gives
\begin{equation*}
 \mathfrak{C}_{-}(t\Omega,\gamma) \le  C_1(m) C_2(m)^{-m/2}\bigg(\frac{\gamma}{1-\gamma}\bigg)^{m/2}T(t\Om)^{m/(m+2)}.
   \end{equation*}
By \eqref{eee1},
\begin{equation*}
(1+\delta)|B_1| \le  C_1(m) C_2(m)^{-m/2}\bigg(\frac{\gamma}{1-\gamma}\bigg)^{m/2} T(t\Om)^{m/(m+2)}.
  \end{equation*}
Finally,
\begin{equation}\label{eq4.1e}
\textup{diam} (t\Om) \le 2r+2 \le (2r+2) \Big (  \frac{ C_1(m)C_2(m)^{-m/2}}{(1+\delta)|B_1|}   \Big) ^{\frac{m+2}{m}}\bigg(\frac{\gamma}{1-\gamma}\bigg)^{\frac{m+2}{2}} T(t\Om).
  \end{equation}
 We observe that the $\gamma$-dependence in both \eqref{eq4.1b} and \eqref{eq4.1e} is the same. Taking the larger of the two $m$-dependant constants  which show up in front of $T(t\Om)$ in \eqref{eq4.1b} and \eqref{eq4.1e}, replacing $t$ from \eqref{eq4.1f}, and letting $\delta \to 0$, and using \eqref{eq4.10} concludes the proof.
\hspace*{\fill }$\square $

\section{Proof of Theorem \ref{the5}\label{sec5}}
The proof of Theorem \ref{the5} requires the extension of the constant $\mathfrak{C}_{-}(\Omega,\gamma)$   to   quasi-open sets.
A proper introduction to the Laplace equation on quasi-open sets, capacity theory, and gamma convergence can be found in \cite[Chapter 2]{HKM93} and \cite{BB05}. We prefer, for expository reasons, to avoid   an extensive   introduction to this topic, and refer the interested reader to \cite[Sections 4.1 and 4.3]{BB05} where all terminology used below can be found.

The key observation is that the class of quasi-open sets is the largest class of sets where the Dirichlet-Laplacian problem is well defined in the Sobolev space $H^1_0$, and satisfies a strong maximum principle (see \cite{FU72}). Of course any open set is also quasi-open. Although the reader may only be interested in open sets, we are forced to work with quasi-open ones since the crucial step of the proof is the existence of a quasi-open set $\Om^*$ which maximises the left-hand side of \eqref{ee11.a}.

 The  strategy of the proof is  as follows. We analyse the shape optimisation problem
\begin{equation}\label{ee11.aq}
 \sup \bigg\{\frac{\mathfrak{C}_{-}(\Omega,\gamma)}{|\Om|} : \Om \sq \R^m, \mbox{ $\Om$  quasi-open with {$0<|\Om|<\infty$}}\bigg\},
\end{equation}
and prove in Step 1 below the existence of a maximiser $\Om^*$. Denoting $C'(m,\gamma)=\mathfrak{C}_{-}(\Omega^*,\gamma)/|\Om^*|$ we  then prove in Step 2  that  $C'(m,\gamma) < \gamma$ by a direct estimate on $\Om^*$.

We start with the following observation. Assume that $(\Om_n)_n$ is a sequence of quasi-open sets of $\R^m$, $|\Om_n|\le 1$, such that $v_{\Om_n}$ converges strongly in $L^2(\R^m)$, and pointwise almost everywhere to some function $v$. Let us denote $\Om:=\{v >0\}$.   We then {have}
\begin{equation}\label{ee10.01}
\mathfrak{C}_{-}(\Omega,\gamma) \ge \limsup _{n \rightarrow +\infty} \mathfrak{C}_{-}(\Omega_n,\gamma).
\end{equation}

Indeed, in order to prove this assertion let us consider a set $A \sq \Om$ such that ${\rm essinf} \; v_{\Om, A, \gamma }<0$. We {have}
\begin{equation}\label{ee10.02}1_\Om(x) \le \liminf_{n \rightarrow +\infty}  1_{\Om_n}(x) \mbox{  a.e. $x \in D$}, \end{equation}
and hence
$$1_{\Om_n \cap A} \rightarrow 1_{A} \mbox{ in } L^1(\R^m).$$

Following \cite[Lemma 4.3.15]{BB05}, there exists larger sets $\tilde \Om_n \supset \Om_n$, $|\tilde \Om_n|\le 2$, such that for a subsequence (still denoted with the same index)
$$\lim_{n \rightarrow +\infty}  v_{ \tilde \Om_n, A\cap \Om_n, \gamma }(x) = v_{\Om, A, \gamma }(x), \mbox{ for a.e. } x\in \R^m.$$
Since ${\rm essinf} \; v_{\Om, A, \gamma }<0$, we get for $n$ large enough that ${\rm essinf} \; v_{\tilde \Om_n, A\cap \Om_n \, \gamma }<0$ for $n$ large enough.   Lemma \ref{lem:ext} (which also holds in the class of quasi-open sets) implies that ${\rm essinf} \; v_{ \Om_n, A\cap \Om_n, \gamma }<0$, since the right-hand side equals to $\gamma,\, \gamma>0$ on $\tilde \Om_n \setminus \Om_n$.
Consequently,
$ \mathfrak{C}_{-}(\Omega_n,\gamma)\le |\Om_n\cap A|$. Passing to the limit,
$$\limsup _{n \rightarrow +\infty} \mathfrak{C}_{-}(\Omega_n,\gamma)\le |A|,$$
 which  implies the assertion.

 Let us prove now that the shape optimisation problem \eqref{ee11.aq} has a solution.
In order to prove this result, it is enough to consider a maximising sequence $(\Om_n)$ of  quasi-open, quasi-connected subsets of $\R^m$, with $|\Om_n|=1$. We first notice that the diameters of $\Om_n$ are uniformly bounded, so that up to a translation all of them are subsets of the same ball $B$. This is a consequence of Theorem \ref{the4.1} which by approximation holds as well on quasi-open, quasi-connected sets. Indeed, this is essentially a consequence of \eqref{eq4.1a} which passes to the limit by approximation.

Then, the existence result is immediate from the compact embedding of $H^1_0(B)\rightharpoonup L^2(B)$ and the observation above: there exists a subsequence such that $v_{\Om_n}$ converges strongly in $L^2(\R^m)$ and pointwise almost everywhere to some function $v$.
Taking $\Om^*:=\{v >0\}$, and using the upper semi-continuity result \eqref{ee10.01} together with the lower semicontinuity of the Lebesgue measures coming from \eqref{ee10.02}, we conclude that $\Om^*$ is optimal.

\section{Proof of Theorem \ref{the6}, and further remarks}\label{sec6}

{\it Proof of Theorem \textup{\ref{the6}}.}
For a measurable set $A \subset  \Om$, we denote
$$m(A):={\rm essinf} \; v_{\Om, A, \gamma }.$$
Note that the smoothness of $\partial\Om$ implies that $v_{\Om, A, \gamma }  \in C^{1, \alpha} (\ov \Om)$.

Firstly we extend the shape functional $m$ on the closure of the convex hull of
$$\{ \gamma 1_{\Om\sm A}-(1-\gamma) 1_A : A \subset \Om, |A|=c\}.$$
 Denote by
$${\mathcal F}:= \{ f \in L^\infty (\Om) : -(1-\gamma) \le f \le \gamma, \int_\Om f  =\gamma|\Om|-c\}.
$$
One naturally extends the functional $m$ to the set ${\mathcal F}$ by defining
$v_{\Om,f,\gamma}$ as the solution of $-\Delta v =f$ in $H^1_0(\Om)$. We shall prove in the sequel that the relaxation of the shape optimisation problem \eqref{e55} on the set ${\mathcal F}$ has a solution in ${\mathcal F}$. Precisely, we solve
\begin{equation}\label{e56}
\min \{ m(f) : f \in {\mathcal F}\}.
\end{equation}
Clearly, ${\mathcal F}$ is compact for the weak-$\star$ $L^\infty$-topology, so that we can assume that $(f_n)_n$ is a   minimising  sequence which converges in weak-$\star$ $L^\infty$ to $f$. We know, by the Calderon-Zygmund inequality, that $\big(v_{\Om, f_n ,\gamma}\big)_n$ are uniformly bounded in $W^{2,p}(\Om)$, for every $p <\infty$. In particular, for $p$ large enough, this  implies that $v_{\Om, f_n, \gamma }$ converges uniformly to $v_{\Om, f, \gamma }$. Consequently, this implies that $m(f_n)$ converges to $m(f)$ so that $f$ is a solution to the optimisation problem \eqref{e56}.

  Secondly we prove that there exists some set $A$ such that $f= \gamma 1_{\Om\sm A}-(1-\gamma)1_A$. To prove this we exploit both the concavity property of the map $f \mapsto m(f)$, and the structure of the partial differential equation.  Assume for contradiction that the set
$$A_\vps:= \{ x\in \Om, -(1-\gamma)+\vps \le f(x) \le \gamma-\vps\}$$
has non-zero measure, for some $\vps >0$. Let $A_1, A_2\subset A_\vps$ be two disjoint sets, such that $|A_1|=|A_2|$. We consider the functions $f_1= f+ t1_{A_1}-t1_{A_2},$ and $f_2= f- t1_{A_1}+t1_{A_2}$, for $t \in (-\vps, \vps)$. Then, $f_1, f_2 \in {\mathcal F}$, and by linearity we have
$$v_{\Om, f, \gamma}= \frac 12  v_{\Om, f_1, \gamma}+  \frac 12 v_{\Om, f_2, \gamma}.$$
Consequently,
$$m(f) \ge \frac 12  m(f_1)+ \frac 12 m(f_2),$$
with strict inequality if the point  $x^*$ where $v_{\Om, f, \gamma}$ is minimised also minimises   $v_{\Om, f_1, \gamma}$ and $ v_{\Om, f_2, \gamma}$. Moreover, we have $ v_{\Om, f, \gamma}(x^*)=  v_{\Om, f_1, \gamma}(x^*)= v_{\Om, f_2, \gamma}(x^*)$. We distinguish between two situations: $v_{\Om, f, \gamma}(x^*)=0,$ and $v_{\Om, f, \gamma}(x^*)<0$. If we are in the first situation, then $x^*$ could belong to $\partial \Om$. In this case, for all admissible sets $A$ we have $v_{\Om,A, \gamma} \ge 0$, the minimal value, which is $0$ being attained on $\partial \Om$. In this case, every admissible set $A$ is a solution to the shape optimisation problem.

If we are in the second situation, then necessarily $x^* \in \Om$. By linearity, from  $ v_{\Om, f, \gamma}(x^*)=  v_{\Om, f_1, \gamma}(x^*)$
we get
$$v_{\Om, 1_{A_1},0}(x^*)=  v_{\Om, 1_{A_2}, 0}(x^*).$$
In particular, for every pair of points $x,y \in A_\vps\sm \{x^*\}$  with  density $1$ in $A_\vps$ we get
$$G_\Om(x^*, x)= G_\Om(x^*, y).$$
Since $G_\Om$ is harmonic on $\Om \sm \{x^*\}$, we get that $G_\Om$ is constant in  $\Om \sm \{x^*\}$, in contradiction with the fact  that  it is a fundamental solution.

Finally, this implies that $|A_\vps|=0$ for every $\vps>0$. Hence $f$ is a characteristic function.
\hspace*{\fill }$\square $

\begin{remark}{\rm
Clearly, the solution of the shape optimisation problem above is, in general, not unique. If the minimal value is $0$, then any admissible set $A$ is a solution. If the minimal value is strictly negative,  then there are geometries with non uniqueness. For example if $\Om$ is the union of two disjoint balls with the same radius, then $A$ is a subset of one of the two balls.

}
\end{remark}

\begin{remark}{\rm
 Assume $\Om=B_R$,  and $|B_R|\ge c\ge \gamma |B_R|$. The solution to the shape optimisation problem \eqref{e55} is given by the (concentric) ball $B_{r_c}$, of mass $c$,  $c=|B_{r_c}|$. Indeed, there are two possibilities. This follows directly from Talenti's theorem applied to $ -v_{{B_R}, A, \gamma}$  in case $A \subset B_R$ has measure $c$ and $v_{B_R, A, \gamma} \le 0$.

Assume now that $v_{B_R, A,\gamma}$ changes sign on $B_R$. We define the sets $\Om^+ =\{ v_{B_R, A,\gamma} >0\}$ and $\Om^- =\{ v_{B_R, A,\gamma} <0\}$. In view of Theorem \ref{the1}, we have  that $|A \cap \Om^+|\le \gamma |\Om^+|$ and $|A \cap \Om^-|\ge \gamma |\Om^-|$. We use Talenti's theorem on $\Om^-$, and get that the essential infimum of the function $v_{B_{R'}, B_{r'}, \gamma}$ is not larger than the infimum of $v_{B_R,A, \gamma}$, where $B_{R'},B_{r'}$ are the balls centred at the origin of measures $|\Om^-|$, $| \Om^- \cap A|$, respectively. We claim that $v_{B_R, B_{r_c},\gamma} \le v_{B_{R'},B_{r'}, \gamma}$. Indeed, making a suitable rescaling by a factor $t \ge1$ such that $|t(\Om^-\cap A)| = c \ge \gamma|B_R|$, the function $v_{tB_{R'}, tB_{r'}, \gamma}$ has an essential infimum lower  than  that of   $v_{B_{R'}, B_{r'},\gamma}$.   We finally notice that
$v_{B_R, B_{r_c}, \gamma}  \le v_{tB_{R'}, B_{r_c}, \gamma}$. Indeed, this is a consequence of the fact that $v_{tB_{R'}, B_{r_c}, \gamma}$ is equal to $\min \{-\delta, v_{B_R, B_{r_c}, \gamma} \} +\delta$, for a suitable $\delta >0$.

}
\end{remark}

\begin{remark}{\rm

Assume  $\Om=B_R$.  Let $0<c < |B_R|$ and denote by $B_{r_c}$ the ball with the same centre as $B_R$ and of volume $c$. For every  radial set $A$ of volume $c$ we have
 $$v_{B_R, B_{r_c}, \gamma }\le v_{B_R,A,\gamma}.$$
    Indeed, let us denote for simplicity $v= v_{B_R,A,\gamma}$ and $v_c= v_{B_R, B_{r_c}, \gamma}$. Using the fact  that both $v$ and $v_c$ are radial, {we get}
 \begin{align*}
 -r^{m-1} v'(r)& = \int_0^r  s^{m-1} (\gamma1_{B_R \setminus A} -(1-\gamma ) 1_A) ds\nonumber \\ &= \frac {1}{\omega_{m-1}} \int_{B_r} \big(\gamma 1_{B_R \setminus A} -(1-\gamma ) 1_A\big),
 \end{align*}
 \begin{align*}-r^{m-1} v_c'(r) &= \int_0^r s^{m-1} (\gamma 1_{B_R \setminus B_{r_c}} -(1-\gamma ) 1_{B_{r_c}}) ds\nonumber \\ &= \frac {1}{\omega_{m-1}} \int_{B_r} \big(\gamma 1_{B_R \setminus B_{r_c}} -(1-\gamma ) 1_{B_{r_c}}\big),
 \end{align*}
 where, for a radial set $E\sq B_R$, we define \textup{(}with abuse of notation\textup{)}, $1_E(r)$ being the value of $1_E$ on the sphere of radius $r$.

Since for all $r\in(0,R)$,
$$\int_{B_r} \big(\gamma 1_{B_R \setminus A} -(1-\gamma ) 1_A\big) \ge \int_{B_r} \big(\gamma 1_{B_R \setminus B_{r_c}} -(1-\gamma ) 1_{B_{r_c}}\big) ,$$
we get that for all $r\in(0,R)$
$$-r^{m-1} v'(r)\ge -r^{m-1} v_c'(r).$$
Hence
$$\int_r^1 v_c'(s) ds \ge \int_r^1 v'(s) ds,$$
and
$$-v_c(r)\ge -v(r).$$
This concludes the proof. Moreover, the infimum value of $v_c$  is attained either at $0$ or at $R$, as $v_c'$ is positive on some interval $(0,\alpha)$ and negative on $(\alpha, R)$.

For $\gamma =\frac 12$, we can compute the value of $c$ such that $v_c(0)=v_c(R)=0$. Indeed, in $\R^2$, the corresponding value $r_c$ is  the  solution of
 $$\frac{r^2}{2}-\frac 14-r^2 \ln r =0.$$
 An estimate of the solution is $r_c\approx 0.432067$.

}
\end{remark}
 \begin{remark}\label{rem01}{\rm

 Assume   $\Om=B_R$,   and $ \mathfrak{C}_{-}(B_R,\gamma)<c< \gamma |B_R|$.
 The solution of the shape optimisation  problem  is non-trivial in this case. While we do not know the general solution, we can observe  a  symmetry breaking phenomenon: the solution is not radially symmetric for small values of $c$.

 Let   $\gamma=\frac 12, r_c=0.432$ just below the value computed in the previous remark. Then, for every radial set $A$, the essential infimum of $ v_{B_R,A,\frac 12}$   is equal   to $0$.
 Meanwhile,  there exists a non-radial set $A$ which gives a lower essential infimum. This fact is observed numerically, if for instance the set $A$ is a disc, centred at $(0.52,0)$ of radius $r_c=0.432$. Of course, the fact that in this case the essential infimum is strictly negative can be directly deduced from  estimates of the Poisson formula. In Figures 1 and 2 below, we display the  (rescaled) numerical solutions computed with MATLAB.
 \begin{figure}[ht]
\includegraphics[width=0.3\textwidth]{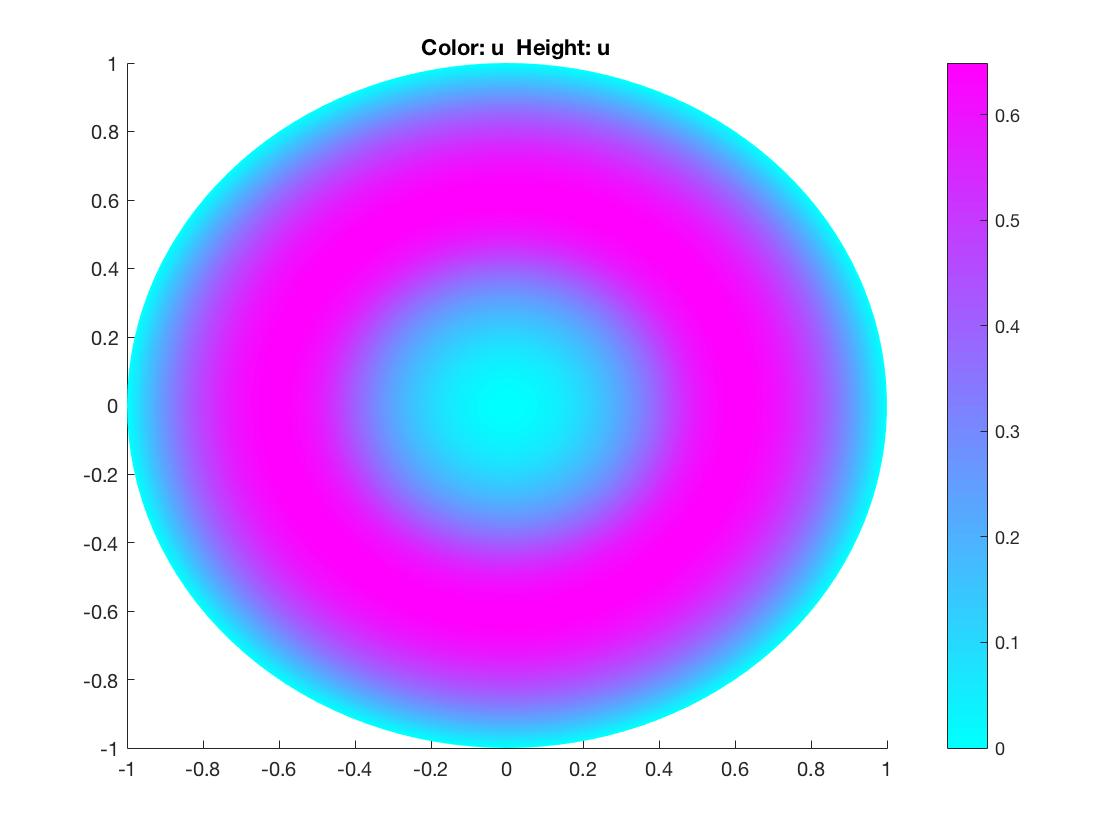}
\includegraphics[width=0.3\textwidth]{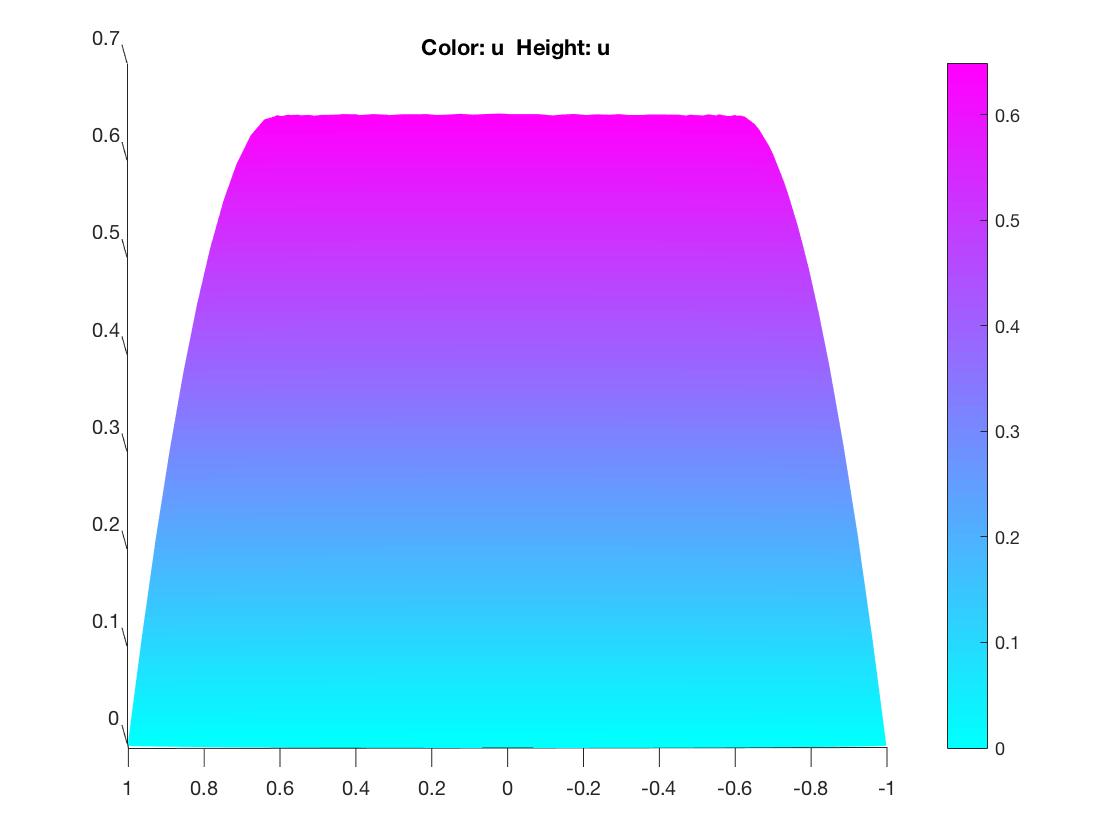}
\includegraphics[width=0.3\textwidth]{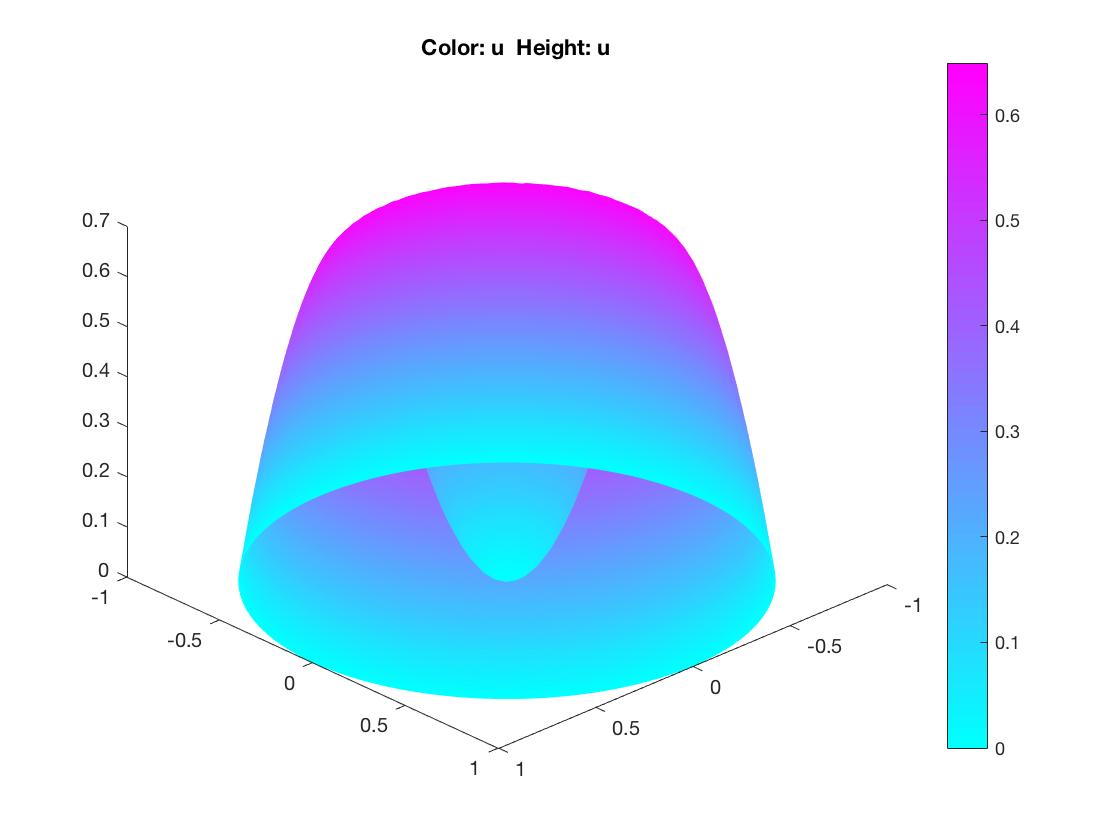}
\caption{Negative mass displayed in the disc centred at $0$ of radius $r=0.432$: the essential infimum is $0$.}
\label{bbp02}
\end{figure}

 \begin{figure}[ht]
\begin{center}
\includegraphics[width=0.3\textwidth]{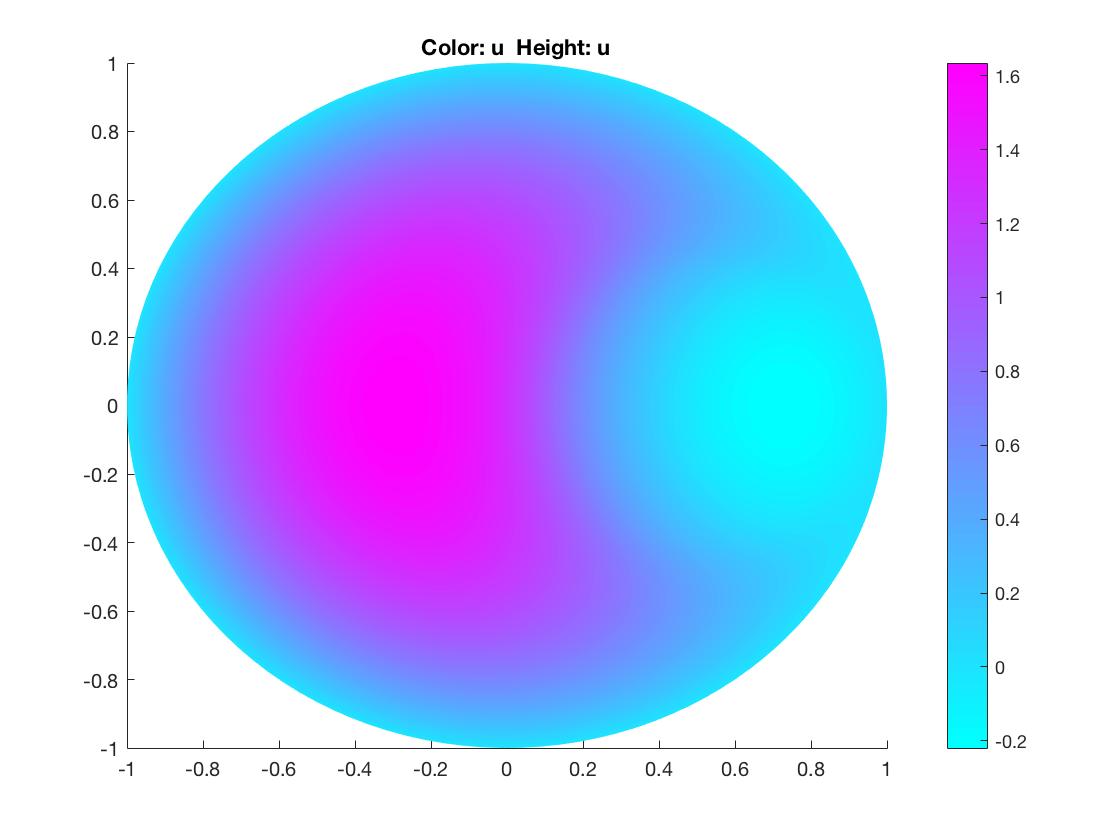}
\includegraphics[width=0.3\textwidth]{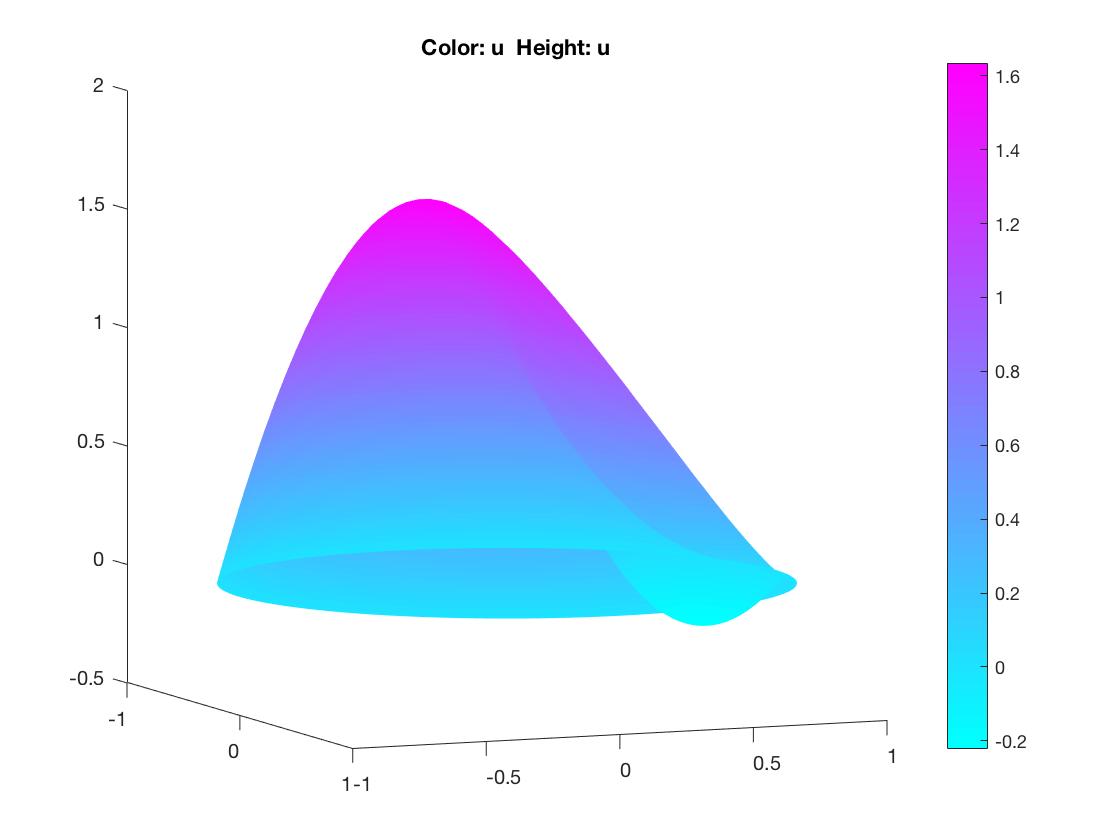}
\caption{Negative mass placed on the disc centred at $(0.52,0)$ of radius $r=0.432$: the essential infimum is negative.}
\end{center}
\label{bbp02.9}
\end{figure}

If $c$ is less than the critical value, the infimum is equal  to $0$, and is attained for an infinite number of solutions to the shape optimisation problem.

}
\end{remark}

\begin{remark}{\rm
 The solutions of the following shape optimisation problems
\begin{equation*}
\max \Big\{   \int_\Om v_{\Om, A, \gamma }   : A \subset \Om, |A|=c\Big\},
\end{equation*}
and
\begin{equation*}
\min \Big\{   \int_\Om v_{\Om, A, \gamma}   : A \subset \Om, |A|=c\Big\},
\end{equation*}
are immediate. Indeed, we observe that
$$ \int_\Om v_{\Om, A, \gamma } = \gamma  \int_{\Om} v_\Om  - \int _A v_\Om .$$
Hence, the position of the set $A$ is a suitable lower/upper level set of $v_\Om$.
}
\end{remark}

\begin{remark}{\rm  If $|A|\le\mathfrak{C}_{-}(B_R,\gamma)$ then $v_{\Omega,A,\gamma}\ge 0,$ and $\int_\Om v_{\Om,A,\gamma}\ge 0$. Below we improve the bound $|A|\le\big(\frac{\gamma}{4m}\big)^m\omega_mR^m$ in \eqref{e00} for $\int_\Om v_{\Om,A,\gamma}\ge 0$ to hold.

Let $\Om\subset\R^m,m\ge 2,$ be an open set with finite measure and a $C^2,$ $R$-smooth boundary. Let $\gamma>0$ and let $v_{\Om,A,\gamma}$ be the solution of \eqref{e2}. If either $m\ge 3$, and
\begin{equation}\label{ef9}
|A|\le\frac{m}{6(m-1)^2}\gamma\omega_mR^m,
\end{equation}
or $m=2$, and
\begin{equation}\label{ef10}
|A|\le \frac{10+7\sqrt 7}{324}\gamma\pi R^2,
\end{equation}
then $\int_{\Om}v_{\Om,A,\gamma}\ge 0.$

\begin{proof} First consider the case $m\ge 3$. By Lemma \ref{lem1} and the coarea formula, we have for $a>0$ that,
\begin{align}\label{ef1}
\int_{\Om}v_{\Om}&\ge \int_{\{x\in \Om: |x-\bar x|<a\}} dx\,\frac{|x-\bar x|R}{2m}\nonumber \\ &
\ge \int_{[0,a]}d\theta \frac{R\theta}{2m}\mathcal{H}^{m-1}(\partial \Om_{\theta}),
\end{align}
where $\mathcal{H}^{m-1}(\partial \Om_{\theta})$ denotes the $(m-1)$-dimensional Hausdorff measure of the parallel set $\{x\in \Om: |x-\bar x|=\theta.\}$
It was shown in Lemma 5 in \cite{vdB} that for an open, bounded set $\Om$ with a $C^2$, $R$-smooth boundary,
\begin{equation}\label{ef2}
\mathcal{H}^{m-1}(\partial \Om_{\theta})\ge \bigg(1-\frac{(m-1)\theta}{R}\bigg)\mathcal{H}^{m-1}(\partial\Om),\, \theta\ge 0.
\end{equation}
By \eqref{ef1} and \eqref{ef2} we  {obtain}
\begin{equation*}
\int_{\Om}v_{\Om}\ge \frac{R}{2m}\bigg(\frac{a^2}{2}-\frac{(m-1)a^3}{3R}\bigg)\mathcal{H}^{m-1}(\partial\Om).
\end{equation*}
Optimising over $a$ yields,
\begin{equation}\label{ef4}
\int_{\Om}v_{\Om}\ge \frac{R^3}{12m(m-1)^2}\mathcal{H}^{m-1}(\partial\Om).
\end{equation}
By the isoperimetric inequality we {have}
\begin{equation}\label{ef5}
\mathcal{H}^{m-1}(\partial\Om)\ge m\omega_m^{1/m}|\Om|^{(m-1)/m}.
\end{equation}
Since $\Om$ contains a ball of radius $R$, $|\Om|\ge \omega_mR^m$. Hence, by \eqref{ef4} and \eqref{ef5},
\begin{equation*}
\mathcal{H}^{m-1}(\partial\Om)\ge m\omega_m^{(m-2)/m}R^{m-3}|\Om|^{2/m}.
\end{equation*}
This, together with \eqref{ef4} {yields}
\begin{equation}\label{ef7}
\int_{\Om}v_{\Om}\ge \frac{\omega_m^{(m-2)/m}}{12(m-1)^2}R^m|\Om|^{2/m}.
\end{equation}
By Talenti's theorem,
\begin{align}\label{ef8}
\int_A v_{\Om}&\le \int_{A^*} v_{\Om^*}\nonumber \\ &
=2^{-1}\omega_m\int_{[0,r_A]} dr\big(R_{\Om}^2-r^2\big)r^{m-1}\nonumber \\ &\le (2m)^{-1}\omega_mR_{\Om}^2r_A^m\nonumber \\ &=(2m)^{-1}\omega_m^{-2/m}|\Om|^{2/m}|A|.
\end{align}
By \eqref{e2}, \eqref{ef7}, and  \eqref{ef8} we have
\begin{align}\label{ee1}
\int_{\Om}v_{\Om,A,\gamma}&=\gamma \int_{\Om}v_{\Om}-\int_Av_{\Om}\nonumber \\ &
\ge \frac{\omega_m^{(m-2)/m}}{12(m-1)^2}\gamma R^m|\Om|^{2/m}-(2m)^{-1}\omega_m^{-2/m}|\Om|^{2/m}|A|.
\end{align}
This implies that $\int_{\Om} v_{\Om,A,\gamma}\ge 0$ for all measurable $A\subset \Om$ satisfying \eqref{ef9}.

Next consider the planar case. By Lemma \ref{lem1}, we have for any  $\alpha\in (0,1)${,}
\begin{equation}\label{ee2}
\int_{\{x\in \Om: |x-\bar x|\ge\alpha R\}}v_{\Om}\ge \frac{\alpha R^2}{4}|\{x\in \Om: |x-\bar x|\ge\alpha R\}|.
\end{equation}
By the coarea formula, Lemma \ref{lem1}, and \eqref{ef2}, we {find}
\begin{align}\label{eee2}
\int_{\{x\in \Om: |x-\bar x|<\alpha R\}}v_{\Om}&\ge \frac{R^3}{4}\Big(\frac{\alpha^2}{2}-\frac{\alpha^3}{3}\Big)\mathcal{H}^{1}(\partial\Om).
\end{align}
By Lemma 5 in \cite{vdB},
\begin{equation}\label{ee8}
\mathcal{H}^{  1 }(\partial \Om_{\theta})\le   \frac{R}{R-\theta} \mathcal{H}^{1}(\partial\Om),\, 0\le\theta<R.
\end{equation}
By the coarea formula, and \eqref{ee8}, we {find}
\begin{align}\label{ee4}
|\{x\in \Om: |x-\bar x|\le\alpha R\}|&  \le  \mathcal{H}^{1}(\partial\Om)\int_{[0,\alpha R]}  d\theta\,\frac{R}{R-\theta}  &\nonumber \\ &\le
\alpha(1-\alpha)^{-1}\mathcal{H}^{1}(\partial\Om)R.
\end{align}
Putting \eqref{ee2}, \eqref{eee2}, and \eqref{ee4} together gives 
\begin{align*}
\int_{\Om}v_{\Om}&\ge \frac{\alpha R^2}{4}|\{x\in \Om: |x-\bar x|\ge\alpha R\}|\nonumber \\ &\hspace{4mm}+\frac{1}{24}\alpha(1-\alpha)(3-2\alpha)R^2|\{x\in \Om: |x-\bar x|\le\alpha R\}|
\nonumber \\ & \ge \min \Big\{\frac{\alpha}{4},\frac{1}{24}\alpha(1-\alpha)(3-2\alpha)\Big\}R^2|\Om|\nonumber \\ &
=\frac{1}{24}\alpha(1-\alpha)(3-2\alpha)R^2|\Om|.
\end{align*}
We choose $\alpha=\frac16(5-\sqrt 7)$  so  as to maximise the above right-hand side, and {obtain}
\begin{equation}\label{ee6}
\int_{\Om}v_{\Om}\ge \frac{10+7\sqrt 7}{1296}R^2|\Om|.
\end{equation}
Formula \eqref{ef8} for $m=2$, \eqref{ee6}, and the first equality in \eqref{ee1} yield,
\begin{equation*}
\int_{\Om}v_{\Om,A,\gamma}\ge \Big(\frac{10+7\sqrt 7}{1296} \gamma R^2-\frac{1}{4\pi}|A|\Big)|\Om|.
\end{equation*}
The above right-hand side is non-negative for all measurable $A\subset \Om$ satisfying \eqref{ef10}.
\end{proof}

}\end{remark}

\bigskip
\bigskip

\noindent
{\bf Acknowledgments.} Both authors were supported by the London Mathematical Society, Scheme 4 grant 41719. MvdB was supported by The Leverhulme Trust through Emeritus Fellowship EM-2018-011-9. DB was supported by the ``Geometry and Spectral Optimization'' research programme
LabEx PERSYVAL-Lab GeoSpec (ANR-11-LABX-0025-01), ANR Comedic (ANR-15-CE40-0006), ANR SHAPO (ANR-18-CE40-0013) and Institut Universitaire de France. The authors are grateful to Beniamin Bogosel for discussions and independent numerical computations related to the assertion of Remark \ref{rem01}.


\begin{thebibliography} {99}



\bibitem{ALT89} A. Alvino, G. Trombetti, P.-L. Lions,
On optimisation problems with prescribed rearrangements.
Nonlinear Anal. 13 (1989), 185--220.


\bibitem{vdB}M. van den Berg, On the asymptotics of the heat equation and bounds on traces associated with the Dirichlet Laplacian. J. Funct. Anal. 71 (1987), 279--293.

\bibitem{vdBC}M. van den Berg, T. Carroll, Hardy inequality and $L^p$ estimates for the torsion function. Bull. Lond. Math. Soc. 41 (2009), 980--986.


\bibitem{MvdBB}M. van den Berg, D. Bucur, On the torsion function with Robin or Dirichlet boundary conditions. J. Funct. Anal. 266 (2014), 1647--1666.


\bibitem{Br14} L. Brasco, On torsional rigidity and principal frequencies: an invitation to the Kohler-Jobin rearrangement technique.
ESAIM Control Optim. Calc. Var. 20 (2014), 315--338.




\bibitem{BB05} D. Bucur, G. Buttazzo,  Variational methods in shape optimization problems. Progress in Nonlinear Differential Equations and their Applications, 65. Birkh\"auser Boston, Inc., Boston, MA (2005).

\bibitem{BBV}D. Bucur, G. Buttazzo, B. Velichkov, Spectral optimisation problems for potentials and measures. SIAM J. Math. Anal. 46 (2014), 2956--2986.

\bibitem{BT11}
G. R. Burton, J. F. Toland,
Surface waves on steady perfect-fluid flows with vorticity.
Comm. Pure Appl. Math. 64 (2011), 975--1007.

\bibitem{CL96}
S. Cox, R. Lipton,
Extremal eigenvalue problems for two-phase conductors.
Arch. Rational Mech. Anal. 136 (1996), 101--117.

\bibitem{EBD3}E. B. Davies, Heat kernels and spectral theory. Cambridge University Press, Cambridge (1989).

\bibitem{FH}S. Fournais, B. Helffer, Inequalities for the lowest magnetic Neumann eigenvalue. {Lett. Math. Phys. 109 (2019), 1683--1700.}

\bibitem{FU72} B. Fuglede, Finely harmonic functions, Lecture Notes in Mathematics 289 Springer, Berlin,
Heidelberg, New York (1972).

\bibitem{GS}T. Giorgi, R. G. Smits, Principal eigenvalue estimates
via the supremum of torsion. Indiana Univ. Math. J. 59 (2010),
987--1011.

\bibitem{GB}A. Grigor'yan, Analytic and geometric backgroud of recurrence and non-explosion of the Brownian motion on
Riemannian manifolds. Bull. Amer. Math. Soc. (N.S.) 36 (1999), 135--249.

\bibitem{GB1}A. Grigor'yan, Heat kernel and Analysis on
manifolds. AMS-IP Studies in Advanced Mathematics, \textbf{47},
American Mathematical Society, Providence, RI; International
Press, Boston, MA (2009).

\bibitem{HKM93}J. Heinonen, T. Kilpel\"ainen, O. Martio, Nonlinear potential theory of degenerate elliptic equations. Oxford Mathematical Monographs. Oxford Science Publications. The Clarendon Press, Oxford University Press, New York (1993).

\bibitem{HKS} B. Helffer, H. Kovarik, M. P. Sundqvist On the semi-classical analysis of the groundstate energy of the Dirichlet Pauli operator III: Magnetic fields that change sign. {Lett. Math. Phys. 109 (2019), 1533--1558.}

\bibitem{LLNP16} J. Lamboley, A. Laurain, G. Nadin, Y. Privat,
Properties of optimisers of the principal eigenvalue with indefinite weight and Robin conditions.
Calc. Var. Partial Differential Equations 55 (2016), no. 6, Art. 144, 37 pp.

\bibitem{Li83} E. H. Lieb, On the lowest eigenvalue of the Laplacian for the intersection of two domains.
Invent. Math. 74 (1983), 441--448.


\bibitem{ke06}S. Kesavan,
Symmetrization \& applications.
Series in Analysis, 3. World Scientific Publishing Co. Pte. Ltd., Hackensack, NJ (2006).

\bibitem{McG}I. McGillivray, An unstable two-phase membrane problem and maximum flux exchange flow. Appl. Math. Optim. 75 (2017), 365--401.

\bibitem{TG}S. Timoshenko, J. N. Goodier, Theory of elasticity, McGraw-Hill Book Company, Inc.  New York (1951).

\bibitem{Tr88} N. Trudinger,
Comparison principles and pointwise estimates for viscosity solutions of nonlinear elliptic equations.
Rev. Mat. Iberoamericana 4 (1988), 453--468.

\bibitem{HV}H. Vogt, $L_{\infty}$ estimates for the torsion function and $L_{\infty}$ growth of semigroups satisfying Gaussian bounds. {Potential Analysis 51 (2019), 37--47.}

\end{thebibliography}
\end{document}